\UseRawInputEncoding
\documentclass{amsart}
\usepackage{mathrsfs}
\usepackage{graphicx}
\usepackage{amsmath}
\usepackage{amscd}
\usepackage{cite}
\usepackage{hyperref}
\usepackage{epsfig}
\usepackage{subfigure}
\usepackage{caption}
\usepackage{tikz}
  \usepackage{float}
  \usepackage{amssymb}
\usepackage{amsfonts}
\usepackage[all,cmtip]{xy}
\usepackage{appendix}

\newtheorem{Example}{Example}[section]
\newtheorem{Definition}{Definition}[section]
\newtheorem{Theorem}{Theorem}[section]
\newtheorem{Theorem/Definition}{Theorem/Definition}[section]
\newtheorem{Proposition}{Proposition}[section]
\newtheorem{Lemma}{Lemma}[section]

\newcommand{\bK}{{\mathbb K}}

\newcommand{\bQ}{{\mathbb Q}}
\newcommand{\bR}{{\mathbb R}}
\newcommand{\bZ}{{\mathbb Z}}

\newcommand{\cF}{{\mathcal F}}
\newcommand{\cG}{{\mathcal G}}

\newcommand{\cM}{{\mathcal M}}

\newcommand{\half}{\frac{1}{2}}
\newcommand{\cV}{{\mathcal V}}

\newcommand{\Mbar}{\overline{\cM}}

\newcommand{\wF}{{\widehat F}}
\newcommand{\wcF}{{\widehat{\mathcal{F}}}}
\newcommand{\tF}{{\widetilde F}}

\newcommand{\be}{\begin{equation}}
\newcommand{\ee}{\end{equation}}
\newcommand{\bea}{\begin{eqnarray}}
\newcommand{\ben}{\begin{eqnarray*}}
\newcommand{\een}{\end{eqnarray*}}
\newcommand{\eea}{\end{eqnarray}}

\DeclareMathOperator{\Aut}{Aut}
\DeclareMathOperator{\Id}{id}
\DeclareMathOperator{\ext}{ext}
\DeclareMathOperator{\val}{val}
\DeclareMathOperator{\genus}{genus}
\DeclareMathOperator{\Ver}{Ver}

\usepackage{color}

\definecolor{yellow}{rgb}{1,1,0}
\definecolor{orange}{rgb}{1,.7,0}
\definecolor{red}{rgb}{1,0,0} \definecolor{green}{rgb}{0,1,1}
\definecolor{white}{rgb}{1,1,1}

\definecolor{A}{rgb}{.75,1,.75}

\theoremstyle{remark}
\newtheorem{Remark}{Remark}[section]

\begin{document}

\newtheorem{myDef}{Definition}
\newtheorem{thm}{Theorem}
\newtheorem{eqn}{equation}

\title[M\"obius Inversion and Duality for Summations of Stable Graphs]
{M\"obius Inversion and Duality for Summations of Stable Graphs}

\author{Zhiyuan Wang}
\address{School of Mathematics and Statistics\\
Huazhong University of Science and Technology\\Wuhan, 430074, China}
\email{wangzy23@hust.edu.cn}

\author{Jian Zhou}
\address{Department of Mathematical Sciences\\
Tsinghua University\\Beijing, 100084, China}
\email{jianzhou@mail.tsinghua.edu.cn}

\begin{abstract}
Using the stratifications of Deligne-Mumford moduli spaces $\overline{\mathcal M}_{g,n}$
indexed by stable graphs,
we introduce a partially ordered set of stable graphs
by defining a partial ordering on the set of connected stable graphs
of genus $g$ with $n$ external edges.
By modifying the usual definition of zeta function and M\"obius function of a poset,
we introduce generalized ($\mathbb Q$-valued)
zeta function and generalized ($\mathbb Q$-valued) M\"obius function of
the poset of stable graphs.
We use them to proved a generalized M\"obius inversion formula for functions
on the poset of stable graphs.
Two applications related to duality in earlier work are also presented.
\end{abstract}

\maketitle

\section{Introduction}

As the title indicates,
the subject of this work lies in the realm of building some bridges
between some techniques in combinatorics and some techniques in mathematical physics.
The theory of M\"obius inversions
is the subject of the first and the sixth papers in the famous series entitled ``On the foundations
of combinatorial theory" by Rota and his collaborators \cite{ro, drs}.
Stable graphs are originally used in algebraic geometry to describe
the boundary strata of the Deligne-Mumford moduli spaces of algebraic curves \cite{dm, kn}.
See also \cite[Section XII.10]{ACGH2}.
Later they have also been widely used in string theory.
Usually in quantum field theory,
physically important objects are often expressed as a summation over some
Feynman graphs with some specific Feynman rules.
Because algebraic curves are indispensable in string theory,
it is not unexpected that in string theory summations over stable graphs
with suitable Feynman rules are useful.
See for example \cite{bcov2}.
In the reversed direction \cite{bh},
summations over stable graphs with Feynman rules
specified by the Harer-Zagier formula for the orbifold Euler characteristics $\cM_{g,n}$
have been used to express the orbifold Euler characteristics of $\Mbar_{g,n}$.

The connection between M\"obius inversion and summation of stable graphs is
an unexpected byproduct of our earlier work on the formalism  of abstract quantum field
theories based on summation of stable graphs.
In many examples in the literature
partition functions or $n$-point correlation functions are expressed as summation over
stable graphs, whereas difference between different theories is just their different
Feynman rules. Furthermore, these partition functions or correlation  functions
are often computed by similar recursion relations.
These fact inspire us to introduce the formalism of ``abstract" quantum field theory
based on stable graphs,
meaning we define abstract partition functions and correlation functions
as formal sums {\em of} stable graphs without specifying any Feynman rule.
We then use simple operations on the set of stable graphs such as edge cutting
to derive some universal recursion relations.
When a concrete theory involving summations over stable graphs
can be obtained by specifying some concrete Feynman rule,
and the recursion relations in this theory can be derived from the
universal recursion relations,
we say this theory is a {\em realization} of the abstract theory.

One of the applications of this formalism is the holomorphic anomaly equation \cite{bcov1, bcov2} as reformulated
by \cite{abk}.
To apply our formalism to this setting in the most general case,
it turns out to be necessary to develop a duality theory of stable graphs \cite{wz2}.
Another application of our formalism is the computations of the orbifold Euler characteristics $\chi(\Mbar_{g,n})$
from the Harer-Zagier formula for $\chi(\cM_{g,n})$.
In this setting we have found an inversion formula that reversely expresses
$\chi(\cM_{g,n})$ in terms of $\chi(\Mbar_{g,n})$.
This can be understood as an open-closed duality in string theory.

The purpose of this paper is to interpret such dualities
in terms of M\"obius inversion formula on locally finite partially ordered sets
as developed in \cite{ro, drs}.
For this purpose,
we introduce a partial ordering on the set of stable graphs
by edge contraction.
Next we modify the usual definition of the zeta-function and the M\"obius function
 on locally finite poset  to define generalized zeta-function
 and the M\"obius function for the poset of stable graphs.
The reason why the modifications are necessary is due to the fact that
when we deal with stable graphs
we are actually working with isomorphism classes of stable graphs,
and hence when an enumeration problem is involved,
it is more suitable to take into account the orders of automorphism groups.
Using these  generalizations we then prove the generalized M\"obius inversion
formula for the poset of stable graphs.

We arrange the rest of the paper as follows.
In Section \ref{sec-pre} we recall the notion of stable graphs and its relation with Deligne-Mumford
moduli spaces.
We also recall our formalism of abstract quantum field theory of summation of stable graphs
and its duality theory.
In Section \ref{sec-inversion-Mob} we first recall the theory of M\"obius inversion
on locally finite posets.
Then we introduce a partial ordering on the set of isomorphism classes of stable graphs.
We next define  generalized zeta function, generalized M\"obius function,
and prove the generalized M\"obius inversion formula, for this poset.
The remaining two sections deal with applications.
The application to the duality of stable graphs developed in \cite{wz2} is
presented in Section \ref{realization-duality}.
The application to the open-closed duality of orbifold Euler characteristics
of $\cM_{g,n}$ and $\Mbar_{g,n}$ is presented in Section \ref{sec-O-C}.

\section{Preliminaries of Duality in Abstract QFT for Stable Graphs}
\label{sec-pre}

In this section,
we recall some preliminaries of
the abstract quantum field theory for stable graphs developed in \cite{wz}
and the duality theory for stable graphs developed in \cite{wz2}.

\subsection{Stable graphs and stratification of $\Mbar_{g,n}$}
\label{sec-pre-absqft}

In this subsection we recall the definition of stable graphs.

A stable graph $\Gamma$ consists of the following data:
\begin{itemize}
\item[1)]
A set of vertices $V(\Gamma)$,
together with a non-negative number $g_v\in \bZ_{\geq 0}$
(called `genus' of this vertex)
associated to each vertex $v\in V(\Gamma)$.
\item[2)]
A set of internal edges $E(\Gamma)$
connecting the vertices.
If the two endpoints of $e\in E(\Gamma)$ is attached to a same vertex,
then this edge is called a loop.
\item[3)]
A set of external edges $E^{\ext}(\Gamma)$.
Each external edges is attached to a vertex.
\end{itemize}

A half-edge is either an external edge,
or an internal edge together with a choice of an end point.
We will denoted by $H(\Gamma)$ the set of all half-edges in $\Gamma$,
and by $H(v)$ the set of all half-edges attached to the vertex $v\in V(\Gamma)$.
The number $\val_v:= |H(v)|$ will be called the valence of $v$.
Then a stable graph is defined to be a graph satisfying the following stability condition:
\be
\label{eq-stable-def}
2g_v -2 + \val_v >0,
\qquad\forall v\in V(\Gamma).
\ee

The genus of a connected stable graph $\Gamma$ is defined to be:
\ben
\genus(\Gamma):= h^1(\Gamma)+ \sum_{v\in V(\Gamma)}g_v,
\een
where $h^1(\Gamma)$ is the number of independent loops in $\Gamma$.
Furthermore,
if $\Gamma=\Gamma_1\sqcup\cdots\sqcup \Gamma_k$ is a disconnected graph
where $\Gamma_1,\cdots,\Gamma_k$ are the connected components,
then the genus of $\Gamma$ is defined to be:
\be
\label{eq-genus-disconn}
\genus(\Gamma):=\genus(\Gamma_1)+\cdots +\genus(\Gamma_k)-k+1.
\ee
In what follows,
we will denote:
\be
\begin{split}
&\cG_{g,n}:= \{
\text{stable graphs of genus $g$ with $n$ external edges}
\};\\
&\cG_{g,n}^c:= \{
\text{connected stable graphs of genus $g$ with $n$ external edges}
\},
\end{split}
\ee
where $2g-2+n>0$ (by the stability condition).

It is known that stable graphs describe
the stratification of the Deligne-Mumford moduli space $\Mbar_{g,n}$ of stable curves,
see \cite{dm, kn}.
In the rest of this subsection let us briefly recall this stratification.
Let $\cM_{g,n}$ be the moduli space of connected smooth stable curves of genus $g$ with $n$ marked points,
and let $\Mbar_{g,n}$ be the Deligne-Mumford compactification of $\cM_{g,n}$.
Then:
\begin{equation*}
\Mbar_{g,n}=\big\{(\Sigma;x_1,x_2,\cdots,x_n)\big\}/\sim,
\end{equation*}
where $\Sigma$ is a connected algebraic curve whose singular points are nodal points,
with $n$ distinct smooth points $x_1,\cdots,x_n\in \Sigma$ (called marked points) on it,
satisfying the following stability condition:
there are at least three special points (marked points or nodes)
on each irreducible component of genus $0$,
and at least one special point
on each irreducible component of genus $1$.
The equivalence relation $\sim$ between two stable curves
$(\Sigma;x_1,\cdots,x_n)$ and $(\Sigma';x_1',\cdots,x_n')$ is given by
a biholomorphic map
\ben
f:\Sigma\to \Sigma',
\een
such that $f(x_i)=x_i'$ for every $i$.
The moduli space $\Mbar_{g,n}$ is a complex orbifold of dimension $3g-3+n$.
Notice that the symmetric group $S_n$ acts naturally on $\Mbar_{g,n}$
by permuting the markings,
and we will regard $\Mbar_{g,n}/S_n$ as the moduli space
of stable curves on which we do not distinguish the $n$ marked points.

Given a connected stable curve of genus $g$ with $n$ marked points,
one can associate a stable graph in $\cG_{g,n}^c$ in the following way:
\begin{itemize}
\item[1)]
An irreducible component $C$ of genus $g$ $\mapsto$ a vertex $v_C$ of genus $g$;
\item[2)]
A nodal point where the components $C_1$ and $C_2$ intersects
($C_1$, $C_2$ may not be distinct) $\mapsto$
an internal edge connecting $v_{C_1}$ and $v_{C_2}$;
\item[3)]
A marked point on the component $C$ $\mapsto$
an external edge attached to $v_C$.
\end{itemize}
Now given a stable graph $\Gamma$,
denote by $\cM_\Gamma$ the moduli space of stable curves
whose dual graph is $\Gamma$.
Then there is a natural stratification of $\Mbar_{g,n}/S_n$:
\ben
\Mbar_{g,n}/S_n=\bigsqcup_{\Gamma\in\cG_{g,n}^c}\cM_\Gamma.
\een
Here $\cM_\Gamma$ is a locally closed strata of codimension $|E(\Gamma)|$
in $\Mbar_{g,n}/S_n$.

\subsection{Abstract free energies and abstract $n$-point functions}

In this subsection,
we review the construction of the
abstract free energies and abstract $n$-point functions.
They are certain linear combinations of stable graphs,
see \cite[\S 2]{wz}.

First let us recall the automorphisms of stable graphs.
Let $\Gamma$ be a stable graph,
then an automorphism $\phi$ of $\Gamma$ consists of two bijections
$\phi_V:V(\Gamma)\to V(\Gamma)$ and $\phi_H:H(\Gamma)\to H(\Gamma)$,
satisfying:
\begin{itemize}
\item[1)]
For every $v\in V(\Gamma)$,
the vertices $\phi_V(v)$ and $v$ are of the same genus;
\item[2)]
If a half-edge $h$ is attached to $v\in V(\Gamma)$,
then $\phi_H(h)$ is attached to $\phi_V(v)$;
\item[3)]
If $h_1,h_2\in H(\Gamma)$ are the two half-edges of an internal edge,
then so are $\phi_H(h_1)$ and $\phi_H(h_2)$.
\end{itemize}
Given a stable graph $\Gamma$,
denote by $\Aut(\Gamma)$ the set of automorphisms of $\Gamma$.

\begin{Definition}
[\cite{wz}]
\label{def-absnpt}

Let $(g,n)$ be a pair of non-negative integers with $2g-2+n>0$.
Define the abstract $n$-point function $\wcF_{g,n}$
to be the following formal summation of stable graphs:
\be
\label{eq-def-absnpt}
\wcF_{g,n}:=\sum_{\Gamma\in\cG_{g,n}^c}
\frac{1}{|\Aut(\Gamma)|}\Gamma.
\ee
For $g\geq 2$,
the abstract free energy of genus $g$ is defined to be:
\be
\label{eq-def-absfe}
\wcF_g := \wcF_{g,0} = \sum_{\Gamma\in\cG_{g,0}^c}
\frac{1}{|\Aut(\Gamma)|}\Gamma.
\ee
\end{Definition}

The stability condition \eqref{eq-stable-def} ensures that
\eqref{eq-def-absnpt} and \eqref{eq-def-absfe} are finite summations.
Thus $\wcF_{g,n}^c$ is a vector in:
\be
\label{eq-vectorspace-cV}
\cV_{g,n}^c:=
\bigoplus_{\Gamma\in\cG_{g,n}^c}\bQ\Gamma
\quad\subset \quad
\cV_{g,n}:=\bigoplus_{\Gamma\in\cG_{g,n}}\bQ\Gamma.
\ee

\begin{Example}
\label{eg-abstract-fe}
We have:
\begin{flalign*}
\begin{tikzpicture}[scale=0.92]
\node [align=center,align=center] at (0.3,0) {$\widehat{\cF}_{0,3}=\frac{1}{6}$};
\draw (1.6,0) circle [radius=0.2];
\draw (1.1,0)--(1.4,0);
\draw (1.76,0.1)--(2,0.15);
\draw (1.76,-0.1)--(2,-0.15);
\node [align=center,align=center] at (1.6,0) {$0$};
\node [align=center,align=center] at (2.35,-0.15) {$,$};
\end{tikzpicture}&&
\end{flalign*}
\begin{flalign*}
\begin{tikzpicture}[scale=0.92]
\node [align=center,align=center] at (0.3,0) {$\widehat{\cF}_{0,4}=\frac{1}{24}$};
\draw (1.6,0) circle [radius=0.2];
\draw (1.1,0.15)--(1.44,0.1);
\draw (1.1,-0.15)--(1.44,-0.1);
\draw (1.76,0.1)--(2.1,0.15);
\draw (1.76,-0.1)--(2.1,-0.15);
\node [align=center,align=center] at (1.6,0) {$0$};
\node [align=center,align=center] at (2.6,0) {$+\frac{1}{8}$};
\draw (1.6+1.9,0) circle [radius=0.2];
\draw (2.2+1.9,0) circle [radius=0.2];
\draw (1.1+1.9,0.15)--(1.44+1.9,0.1);
\draw (1.1+1.9,-0.15)--(1.44+1.9,-0.1);
\draw (1.1+1.9,0.15)--(1.44+1.9,0.1);
\draw (2.36+1.9,0.1)--(2.7+1.9,0.15);
\draw (2.36+1.9,-0.1)--(2.7+1.9,-0.15);
\draw (1.8+1.9,0)--(2+1.9,0);
\node [align=center,align=center] at (1.6+1.9,0) {$0$};
\node [align=center,align=center] at (2.2+1.9,0) {$0$};
\node [align=center,align=center] at (4.95,-0.15) {$,$};
\end{tikzpicture}&&
\end{flalign*}
\begin{flalign*}
\begin{tikzpicture}[scale=0.92]
\node [align=center,align=center] at (0.3-0.1,0) {$\widehat{\cF}_{0,5}=\frac{1}{120}$};
\draw (1.6,0) circle [radius=0.2];
\draw (1.1,0.15)--(1.44,0.1);
\draw (1.1,-0.15)--(1.44,-0.1);
\draw (1.76,0.1)--(2.1,0.15);
\draw (1.76,-0.1)--(2.1,-0.15);
\draw (1.6,0.2)--(1.6,0.4);
\node [align=center,align=center] at (1.6,0) {$0$};
\node [align=center,align=center] at (2.6,0) {$+\frac{1}{8}$};
\draw (1.6+1.9,0) circle [radius=0.2];
\draw (2.2+1.9,0) circle [radius=0.2];
\draw (2.8+1.9,0) circle [radius=0.2];
\draw (1.1+1.9,0.15)--(1.44+1.9,0.1);
\draw (1.1+1.9,-0.15)--(1.44+1.9,-0.1);
\draw (1.1+1.9,0.15)--(1.44+1.9,0.1);
\draw (2.96+1.9,0.1)--(3.3+1.9,0.15);
\draw (2.96+1.9,-0.1)--(3.3+1.9,-0.15);
\draw (1.8+1.9,0)--(2+1.9,0);
\draw (2.4+1.9,0)--(2.6+1.9,0);
\draw (2.2+1.9,0.2)--(2.2+1.9,0.4);
\node [align=center,align=center] at (1.6+1.9,0) {$0$};
\node [align=center,align=center] at (2.2+1.9,0) {$0$};
\node [align=center,align=center] at (2.8+1.9,0) {$0$};
\node [align=center,align=center] at (5.8,0) {$+\frac{1}{12}$};
\draw (1.6+5.1,0) circle [radius=0.2];
\draw (2.2+5.1,0) circle [radius=0.2];
\draw (1.1+5.1,0.15)--(1.44+5.1,0.1);
\draw (1.1+5.1,-0.15)--(1.44+5.1,-0.1);
\draw (1.1+5.1,0.15)--(1.44+5.1,0.1);
\draw (2.36+5.1,0.1)--(2.7+5.1,0.15);
\draw (2.36+5.1,-0.1)--(2.7+5.1,-0.15);
\draw (1.8+5.1,0)--(2+5.1,0);
\draw (2.4+5.1,0)--(2.7+5.1,0);
\node [align=center,align=center] at (1.6+5.1,0) {$0$};
\node [align=center,align=center] at (2.2+5.1,0) {$0$};
\node [align=center,align=center] at (8.2,-0.15) {$,$};
\end{tikzpicture}&&
\end{flalign*}
\begin{flalign*}
\begin{tikzpicture}[scale=0.92]
\node [align=center,align=center] at (0.3-0.2,0) {$\widehat{\cF}_{1,1}=$};
\draw (1,0) circle [radius=0.2];
\draw (1.2,0)--(1.5,0);
\node [align=center,align=center] at (1,0) {$1$};
\node [align=center,align=center] at (2,0) {$+\frac{1}{2}$};
\draw (1+1.8,0) circle [radius=0.2];
\draw (1.2+1.8,0)--(1.5+1.8,0);
\draw (0.84+1.8,0.1) .. controls (0.5+1.8,0.2) and (0.5+1.8,-0.2) ..  (0.84+1.8,-0.1);
\node [align=center,align=center] at (1+1.8,0) {$0$};
\node [align=center,align=center] at (3.6,-0.15) {$,$};
\end{tikzpicture}&&
\end{flalign*}
\begin{flalign*}
\begin{tikzpicture}[scale=0.92]
\node [align=center,align=center] at (0.3-0.65,0) {$\widehat{\cF}_{1,2}=\frac{1}{2}$};
\draw (1,0) circle [radius=0.2];
\draw (1.2,0)--(1.5,0);
\draw (0.5,0)--(0.8,0);
\node [align=center,align=center] at (1,0) {$1$};
\node [align=center,align=center] at (2,0) {$+\frac{1}{4}$};
\draw (1+1.8,0) circle [radius=0.2];
\draw (1.17+1.8,0.1)--(1.4+1.8,0.15);
\draw (1.17+1.8,-0.1)--(1.4+1.8,-0.15);
\draw (0.84+1.8,0.1) .. controls (0.5+1.8,0.2) and (0.5+1.8,-0.2) ..  (0.84+1.8,-0.1);
\node [align=center,align=center] at (1+1.8,0) {$0$};
\node [align=center,align=center] at (3.7,0) {$+\frac{1}{2}$};
\draw (1+3.9,0) circle [radius=0.2];
\draw (0.4+3.9,0) circle [radius=0.2];
\draw (1.17+3.9,0.1)--(1.4+3.9,0.15);
\draw (1.17+3.9,-0.1)--(1.4+3.9,-0.15);
\draw (0.6+3.9,0)--(0.8+3.9,0);
\node [align=center,align=center] at (1+3.9,0) {$0$};
\node [align=center,align=center] at (0.4+3.9,0) {$1$};
\node [align=center,align=center] at (5.8,0) {$+\frac{1}{4}$};
\draw (1+6.2,0) circle [radius=0.2];
\draw (0.4+6.2,0) circle [radius=0.2];
\draw (1.17+6.2,0.1)--(1.4+6.2,0.15);
\draw (1.17+6.2,-0.1)--(1.4+6.2,-0.15);
\draw (0.6+6.2,0)--(0.8+6.2,0);
\node [align=center,align=center] at (1+6.3-0.1,0) {$0$};
\node [align=center,align=center] at (0.4+6.3-0.1,0) {$0$};
\draw (0.24+6.3-0.1,0.1) .. controls (-0.1+6.3-0.1,0.2) and (-0.1+6.3-0.1,-0.2) ..  (0.24+6.3-0.1,-0.1);
\node [align=center,align=center] at (8.2-0.1,0) {$+\frac{1}{4}$};
\draw (1+8.1-0.1,0) circle [radius=0.2];
\draw (0.5+8.1-0.1,0)--(0.8+8.1-0.1,0);
\draw (1.18+8.1-0.1,0.07)--(1.42+8.1-0.1,0.07);
\draw (1.18+8.1-0.1,-0.07)--(1.42+8.1-0.1,-0.07);
\draw (1.8+8.1-0.1,-0)--(2.1+8.1-0.1,0);
\draw (1.6+8.1-0.1,0) circle [radius=0.2];
\node [align=center,align=center] at (1+8.1-0.1,0) {$0$};
\node [align=center,align=center] at (1.6+8.1-0.1,0) {$0$};
\node [align=center,align=center] at (10.5,-0.15) {$,$};
\end{tikzpicture}&&
\end{flalign*}
\begin{flalign*}
\begin{tikzpicture}[scale=0.92]
\node [align=center,align=center] at (0.1+0.4,0) {$\widehat{\cF}_{2}=$};
\draw (1+0.3,0) circle [radius=0.2];
\node [align=center,align=center] at (1+0.3,0) {$2$};
\node [align=center,align=center] at (1.6+0.2,0) {$+\frac{1}{2}$};
\draw (1+1.4+0.2,0) circle [radius=0.2];
\draw (0.84+1.4+0.2,0.1) .. controls (0.5+1.4+0.2,0.2) and (0.5+1.4+0.2,-0.2) ..  (0.84+1.4+0.2,-0.1);
\node [align=center,align=center] at (1+1.4+0.2,0) {$1$};
\node [align=center,align=center] at (3+0.2,0) {$+\frac{1}{2}$};
\draw (1+2.6+0.2,0) circle [radius=0.2];
\draw (1.2+2.6+0.2,0)--(1.4+2.6+0.2,0);
\draw (1.6+2.6+0.2,0) circle [radius=0.2];
\node [align=center,align=center] at (1+2.6+0.2,0) {$1$};
\node [align=center,align=center] at (1.6+2.6+0.2,0) {$1$};
\node [align=center,align=center] at (5,0) {$+\frac{1}{8}$};
\draw (1+4.8,0) circle [radius=0.2];
\draw (0.84+4.8,0.1) .. controls (0.5+4.8,0.2) and (0.5+4.8,-0.2) ..  (0.84+4.8,-0.1);
\draw (1.16+4.8,0.1) .. controls (1.5+4.8,0.2) and (1.5+4.8,-0.2) ..  (1.16+4.8,-0.1);
\node [align=center,align=center] at (1+4.8,0) {$0$};
\node [align=center,align=center] at (6.6,0) {$+\frac{1}{2}$};
\draw (1+6.8,0) circle [radius=0.2];
\draw (0.4+6.8,0) circle [radius=0.2];
\draw (0.6+6.8,0)--(0.8+6.8,0);
\draw (1.16+6.8,0.1) .. controls (1.5+6.8,0.2) and (1.5+6.8,-0.2) ..  (1.16+6.8,-0.1);
\node [align=center,align=center] at (1+6.8,0) {$0$};
\node [align=center,align=center] at (0.4+6.8,0) {$1$};
\node [align=center,align=center] at (8.6,0) {$+\frac{1}{8}$};
\draw (1+9,0) circle [radius=0.2];
\draw (0.4+9,0) circle [radius=0.2];
\draw (0.6+9,0)--(0.8+9,0);
\draw (1.16+9,0.1) .. controls (1.5+9,0.2) and (1.5+9,-0.2) ..  (1.16+9,-0.1);
\draw (0.24+9,0.1) .. controls (-0.1+9,0.2) and (-0.1+9,-0.2) ..  (0.24+9,-0.1);
\node [align=center,align=center] at (1+9,0) {$0$};
\node [align=center,align=center] at (0.4+9,0) {$0$};
\node [align=center,align=center] at (10.6+0.2,0) {$+\frac{1}{12}$};
\draw (1+10.2+0.2,0) circle [radius=0.2];
\draw (1.2+10.2+0.2,0)--(1.4+10.2+0.2,0);
\draw (1.16+10.2+0.2,0.1)--(1.44+10.2+0.2,0.1);
\draw (1.16+10.2+0.2,-0.1)--(1.44+10.2+0.2,-0.1);
\draw (1.6+10.2+0.2,0) circle [radius=0.2];
\node [align=center,align=center] at (1+10.2+0.2,0) {$0$};
\node [align=center,align=center] at (1.6+10.2+0.2,0) {$0$};
\node [align=center,align=center] at (12.45,-0.15) {$.$};
\end{tikzpicture}&&
\end{flalign*}
\end{Example}

\subsection{Dotted stable vertices and dotted stable graphs}
\label{sec-pre-dotted}

In this subsection we recall the construction of dotted stable graphs.
A dotted stable graph is a stable graph whose vertices and edges are drawn
using dotted lines,
and we use such a graph to represent a linear combination of stable graphs in the usual sense.
The notions such as the genus, valence, stability, and automorphisms for a dotted stable graph
are defined as usual.
See \cite[\S 3]{wz2} for details.

First let us recall the definition of dotted stable vertices.
Given a dotted stable vertex of genus $g$ and valence $n$,
we define it to be the linear combination $n!\cdot \wcF_{g,n} \in \cV_{g,n}^c$
of graphs in the usual sense.

\begin{Example}
We have (see Example \ref{eg-abstract-fe}):
\begin{equation*}
\begin{split}
&\begin{tikzpicture}
\draw [dotted,thick](0,0) circle [radius=0.2];
\node [align=center,align=center] at (0,0) {$0$};
\draw [dotted,thick](-0.5,0)--(-0.2,0);
\draw [dotted,thick](0.16,0.1)--(0.5,0.15);
\draw [dotted,thick](0.16,-0.1)--(0.5,-0.15);
\node [align=center,align=center] at (0.8,0) {$=$};
\draw (0+1.6,0) circle [radius=0.2];
\node [align=center,align=center] at (0+1.6,0) {$0$};
\draw (-0.5+1.6,0)--(-0.2+1.6,0);
\draw (0.16+1.6,0.1)--(0.5+1.6,0.15);
\draw (0.16+1.6,-0.1)--(0.5+1.6,-0.15);
\node [align=center,align=center] at (2.5,-0.15) {$,$};
\end{tikzpicture}\\
&\begin{tikzpicture}
\draw [dotted,thick](1.6-1.6,0) circle [radius=0.2];
\draw [dotted,thick](1.1-1.6,0.15)--(1.44-1.6,0.1);
\draw [dotted,thick](1.1-1.6,-0.15)--(1.44-1.6,-0.1);
\draw [dotted,thick](1.76-1.6,0.1)--(2.1-1.6,0.15);
\draw [dotted,thick](1.76-1.6,-0.1)--(2.1-1.6,-0.15);
\node [align=center,align=center] at (1.6-1.6,0) {$0$};
\node [align=center,align=center] at (0.8,0) {$=$};
\draw (1.6+0.1,0) circle [radius=0.2];
\draw (1.1+0.1,0.15)--(1.44+0.1,0.1);
\draw (1.1+0.1,-0.15)--(1.44+0.1,-0.1);
\draw (1.76+0.1,0.1)--(2.1+0.1,0.15);
\draw (1.76+0.1,-0.1)--(2.1+0.1,-0.15);
\node [align=center,align=center] at (1.6+0.1,0) {$0$};
\node [align=center,align=center] at (2.6,0) {$+3$};
\draw (1.6+1.9,0) circle [radius=0.2];
\draw (2.2+1.9,0) circle [radius=0.2];
\draw (1.1+1.9,0.15)--(1.44+1.9,0.1);
\draw (1.1+1.9,-0.15)--(1.44+1.9,-0.1);
\draw (1.1+1.9,0.15)--(1.44+1.9,0.1);
\draw (2.36+1.9,0.1)--(2.7+1.9,0.15);
\draw (2.36+1.9,-0.1)--(2.7+1.9,-0.15);
\draw (1.8+1.9,0)--(2+1.9,0);
\node [align=center,align=center] at (1.6+1.9,0) {$0$};
\node [align=center,align=center] at (2.2+1.9,0) {$0$};
\node [align=center,align=center] at (5,-0.15) {$,$};
\end{tikzpicture}\\
&\begin{tikzpicture}
\draw [dotted,thick](1-1.6,0) circle [radius=0.2];
\draw [dotted,thick](1.2-1.6,0)--(1.65-1.6,0);
\node [align=center,align=center] at (1-1.6,0) {$1$};
\node [align=center,align=center] at (0.3,0) {$=$};
\draw (1,0) circle [radius=0.2];
\draw (1.2,0)--(1.5,0);
\node [align=center,align=center] at (1,0) {$1$};
\node [align=center,align=center] at (2,0) {$+\frac{1}{2}$};
\draw (1+1.8,0) circle [radius=0.2];
\draw (1.2+1.8,0)--(1.5+1.8,0);
\draw (0.84+1.8,0.1) .. controls (0.5+1.8,0.2) and (0.5+1.8,-0.2) ..  (0.84+1.8,-0.1);
\node [align=center,align=center] at (1+1.8,0) {$0$};
\node [align=center,align=center] at (3.5,-0.15) {$.$};
\end{tikzpicture}
\end{split}
\end{equation*}
\end{Example}

Then we construct the dotted stable graphs by
suitably gluing these dotted stable vertex together along external edges,
and then multiplying by an additional factor $-1$ whenever a new internal edges are obtained.
Here `suitably' means gluing in all possible ways and then taking the average.

More precisely,
Given a dotted stable graph $\Gamma^\vee$,
we associate a linear combinations of graphs in the usual sense in the following way:
\begin{itemize}
\item[1)]
First we distinguish the half-edges of $\Gamma^\vee$ by giving every half-edge a name,
such that different half-edges have different names.
Then we cut off every internal edges of $\Gamma^\vee$ and obtain some dotted stable vertices
whose external edges have distinct names.

\item[2)]
Next,
we express each dotted stable vertex with names on external edges as a summation
over ordinary stable graphs whose external edges have inherited names.
Let $\Gamma'^\vee$ be a dotted stable vertex of genus $g$ and valence $n$
whose external edges have different names,
then we define it to be the following linear combination of graphs (in the usual sense)
whose external edges has names:
\begin{equation*}
\Gamma'^\vee=\sum_{\Gamma'}\frac{1}{|\Aut(\Gamma')|}\Gamma',
\end{equation*}
where the sum is over all possible stable graphs $\Gamma'$ of genus $g$ with $n$ external edges,
such that the external edges have the same $n$ different names as the external edges of $\Gamma'^\vee$.
The automorphisms of such $\Gamma'$ should preserves the names of
all the external edges,
hence they fix each of the external edges.

\item[3)]
Finally we glue the external edges of these ordinary stable graphs together
in such a way that two external edges are glued together if and only if
their ancestors  in $\Gamma^\vee$ are joined together.
Finally, we forget all these names,
and multiply by a factor $(-1)^{|E^\vee(\Gamma^\vee)|}$ to the whole expression,
where $E^\vee(\Gamma^\vee)$ is the set of all dotted internal edges of $\Gamma^\vee$.
In this way we obtain a linear combination
of ordinary stable graphs associated to
the dotted stable graph $\Gamma^\vee$.
In the following we will abuse the notations by using $\Gamma^\vee$
to denote both a dotted graph and this linear combination.
\end{itemize}

\begin{Example}
Consider the following dotted graph,
and we give names $a,b,c,d$ to its half-edges then cut the internal edge:
\begin{equation*}
\begin{tikzpicture}
\draw [dotted,thick](1.6-1.6,0) circle [radius=0.2];
\draw [dotted,thick](-0.16,0.1) .. controls (-0.5,0.2) and (-0.5,-0.2) ..  (-0.16,-0.1);
\draw [dotted,thick](1.76-1.6,0.1)--(2.1-1.6,0.15);
\draw [dotted,thick](1.76-1.6,-0.1)--(2.1-1.6,-0.15);
\node [align=center,align=center] at (1.6-1.6,0) {$0$};
\draw [dotted,thick](1.6-1.6+3,0) circle [radius=0.2];
\draw [dotted,thick](-0.16+3,0.1) .. controls (-0.5+3,0.2) and (-0.5+3,-0.2) ..  (-0.16+3,-0.1);
\draw [dotted,thick](1.76-1.6+3,0.1)--(2.1-1.6+3,0.15);
\draw [dotted,thick](1.76-1.6+3,-0.1)--(2.1-1.6+3,-0.15);
\node [align=center,align=center] at (1.6-1.6+3,0) {$0$};
\node [above,align=center] at (-0.16+3,0.1,0) {$a$};
\node [below,align=center] at (-0.16+3,-0.1,0) {$b$};
\node [align=center,right] at (2.1-1.6+3,0.15) {$c$};
\node [align=center,right] at (2.1-1.6+3,-0.15) {$d$};
\draw [dotted,thick](1.6-1.6+6,0) circle [radius=0.2];
\draw [dotted,thick](5.84,0.1)--(5.5,0.15);
\draw [dotted,thick](5.84,-0.1)--(5.5,-0.15);
\draw [dotted,thick](1.76-1.6+6,0.1)--(2.1-1.6+6,0.15);
\draw [dotted,thick](1.76-1.6+6,-0.1)--(2.1-1.6+6,-0.15);
\node [align=center,align=center] at (1.6-1.6+6,0) {$0$};
\node [align=center,left] at (5.5,0.15) {$a$};
\node [align=center,left] at (5.5,-0.15) {$b$};
\node [align=center,right] at (2.1-1.6+6,0.15) {$c$};
\node [align=center,right] at (2.1-1.6+6,-0.15) {$d$};
\node [align=center,right] at (1.3,0) {$\to$};
\node [align=center,right] at (4.2,0) {$\to$};
\node [align=center,right] at (7,-0.2) {$.$};
\end{tikzpicture}
\end{equation*}
Then by definition, the third graph above equals
\begin{equation*}
\begin{tikzpicture}
\draw (0,0) circle [radius=0.2];
\draw (-0.16,0.1)--(-0.5,0.15);
\draw (-0.16,-0.1)--(-0.5,-0.15);
\draw (0.16,0.1)--(0.5,0.15);
\draw (0.16,-0.1)--(0.5,-0.15);
\node [align=center,align=center] at (0,0) {$0$};
\node [align=center,left] at (-0.5,0.15) {$a$};
\node [align=center,left] at (-0.5,-0.15) {$b$};
\node [align=center,right] at (0.5,0.15) {$c$};
\node [align=center,right] at (0.5,-0.15) {$d$};
\node [align=center,align=center] at (1.1,0) {$+$};
\draw (2.2,0) circle [radius=0.2];
\node [align=center,align=center] at (2.2,0) {$0$};
\draw (2.8,0) circle [radius=0.2];
\node [align=center,align=center] at (2.8,0) {$0$};
\draw (2.4,0)--(2.6,0);
\draw (1.7,0.15)--(2.04,0.1);
\draw (1.7,-0.15)--(2.04,-0.1);
\draw (2.96,0.1)--(3.3,0.15);
\draw (2.96,-0.1)--(3.3,-0.15);
\node [align=center,left] at (1.7,0.15) {$a$};
\node [align=center,left] at (1.7,-0.15) {$b$};
\node [align=center,right] at (3.3,0.15) {$c$};
\node [align=center,right] at (3.3,-0.15) {$d$};
\node [align=center,align=center] at (3.9,0) {$+$};
\draw (2.2+2.8,0) circle [radius=0.2];
\node [align=center,align=center] at (2.2+2.8,0) {$0$};
\draw (2.8+2.8,0) circle [radius=0.2];
\node [align=center,align=center] at (2.8+2.8,0) {$0$};
\draw (2.4+2.8,0)--(2.6+2.8,0);
\draw (1.7+2.8,0.15)--(2.04+2.8,0.1);
\draw (1.7+2.8,-0.15)--(2.04+2.8,-0.1);
\draw (2.96+2.8,0.1)--(3.3+2.8,0.15);
\draw (2.96+2.8,-0.1)--(3.3+2.8,-0.15);
\node [align=center,left] at (1.7+2.8,0.15) {$a$};
\node [align=center,left] at (1.7+2.8,-0.15) {$c$};
\node [align=center,right] at (3.3+2.8,0.15) {$b$};
\node [align=center,right] at (3.3+2.8,-0.15) {$d$};
\node [align=center,align=center] at (3.9+2.8,0) {$+$};
\draw (2.2+2.8+2.8,0) circle [radius=0.2];
\node [align=center,align=center] at (2.2+2.8+2.8,0) {$0$};
\draw (2.8+2.8+2.8,0) circle [radius=0.2];
\node [align=center,align=center] at (2.8+2.8+2.8,0) {$0$};
\draw (2.4+2.8+2.8,0)--(2.6+2.8*2,0);
\draw (1.7+2.8+2.8,0.15)--(2.04+2.8*2,0.1);
\draw (1.7+2.8+2.8,-0.15)--(2.04+2.8*2,-0.1);
\draw (2.96+2.8+2.8,0.1)--(3.3+2.8*2,0.15);
\draw (2.96+2.8+2.8,-0.1)--(3.3+2.8+2.8,-0.15);
\node [align=center,left] at (1.7+2.8+2.8,0.15) {$a$};
\node [align=center,left] at (1.7+2.8+2.8,-0.15) {$d$};
\node [align=center,right] at (3.3+2.8+2.8,0.15) {$b$};
\node [align=center,right] at (3.3+2.8+2.8,-0.15) {$c$};
\node [align=center,right] at (9.3,-0.2) {$.$};
\end{tikzpicture}
\end{equation*}
Now we glue the edges named by $a$ and $b$ together,
and then forget the names and multiply by $(-1)^1$.
In this way we obtain:
\begin{equation*}
\begin{tikzpicture}
\draw [dotted,thick](1.6-1.6-0.2,0) circle [radius=0.2];
\draw [dotted,thick](-0.16-0.2,0.1) .. controls (-0.5-0.2,0.2) and (-0.5-0.2,-0.2) ..  (-0.16-0.2,-0.1);
\draw [dotted,thick](1.76-1.6-0.2,0.1)--(2.1-1.6-0.2,0.15);
\draw [dotted,thick](1.76-1.6-0.2,-0.1)--(2.1-1.6-0.2,-0.15);
\node [align=center,align=center] at (1.6-1.6-0.2,0) {$0$};
\node [align=center,align=center] at (0.8,0) {$=-$};
\draw (1.6+0.1,0) circle [radius=0.2];
\draw (1.54,0.1) .. controls (1.2,0.2) and (1.2,-0.2) ..  (1.54,-0.1);
\draw (1.76+0.1,0.1)--(2.1+0.1,0.15);
\draw (1.76+0.1,-0.1)--(2.1+0.1,-0.15);
\node [align=center,align=center] at (1.6+0.1,0) {$0$};
\node [align=center,align=center] at (2.7,0) {$-$};
\draw (1.7+2,0) circle [radius=0.2];
\draw (2.3+2,0) circle [radius=0.2];
\draw (3.54,0.1) .. controls (3.2,0.2) and (3.2,-0.2) ..  (3.54,-0.1);
\draw (2.46+2,0.1)--(2.8+2,0.15);
\draw (2.46+2,-0.1)--(2.8+2,-0.15);
\draw (1.9+2,0)--(2.1+2,0);
\node [align=center,align=center] at (1.7+2,0) {$0$};
\node [align=center,align=center] at (2.3+2,0) {$0$};
\node [align=center,align=center] at (5.3,0) {$-2$};
\draw (6.2,0) circle [radius=0.2];
\draw (5.7,0)--(6,0);
\draw (6.38,0.07)--(6.62,0.07);
\draw (6.38,-0.07)--(6.62,-0.07);
\draw (7,-0)--(7.3,0);
\draw (6.8,0) circle [radius=0.2];
\node [align=center,align=center] at (6.2,0) {$0$};
\node [align=center,align=center] at (6.8,0) {$0$};
\node [align=center,align=center] at (7.6,-0.15) {$.$};
\end{tikzpicture}
\end{equation*}
\end{Example}

Notice that in the above construction we have given two definitions of a dotted stable vertex.
The first way is to define it to be $n!\cdot \wcF_{g,n}$ directly
(where $g$ and $n$ are the genus and valence of the vertex respectively),
and the second way is to regard it as a dotted stable graph and applying the
procedure of naming and forgetting.
These two definitions actually match with each other since:
\begin{Lemma}
[\cite{wz2}]
Denote by $\Gamma\in \cG_{g,n}^c$ a stable graph (in the usual sense) without names,
and $S_\Gamma$ the set of stable graph $\Gamma'$ with names on external edges,
such that $\Gamma$ can be obtained from $\Gamma'$ by forgetting all the names.
Then we have:
\begin{equation*}
\frac{n!}{|\Aut(\Gamma)|}=
\sum_{\Gamma'\in S_\Gamma}\frac{1}{|\Aut(\Gamma')|}.
\end{equation*}
\end{Lemma}

\subsection{The duality map and duality theorem for stable graphs}
\label{sec-pre-duality}

Now in this subsection we recall the duality theorem for stable graphs.
See \cite{wz2} for details.

Let $\Gamma^\vee$ be a dotted stable graph of genus $g$ with $n$ external edges.
By definition it represents a linear combination of stable graphs
in the usual sense,
and one easily see that
the graphs appearing in this linear combination are all of genus $g$ and with $n$ external edges.
I.e.,
$\Gamma^\vee$ is an element in the vector space $\cV_{g,n}$
(see \eqref{eq-vectorspace-cV}).
Moreover,
if $\Gamma^\vee$ is connected,
then it is an element in $\cV_{g,n}^c$.

\begin{Definition}
[\cite{wz2}]
Given $(g,n)$ with $2g-2+n>0$,
define the duality map
\be
\phi_{g,n}:\cV_{g,n}\to\cV_{g,n},
\ee
to be the linear map which maps a graph in the usual sense to the dotted graph
obtained by simply changing this graph into a dotted stable graph of the same shape.
In particular,
when restricted to $\cV_{g,n}^c$,
it gives a linear map
\be
\phi_{g,n}:\cV_{g,n}^c\to\cV_{g,n}^c.
\ee
\end{Definition}

\begin{Example}
For examples, under the map $\phi_{2,0}$ we have:
\begin{equation*}
\begin{split}
&\begin{tikzpicture}
\draw (0,0) circle [radius=0.2];
\node [align=center,align=center] at (0,0) {$2$};
\node [align=center,align=center] at (0.6,0) {$\mapsto$};
\draw [dotted,thick](1.2,0) circle [radius=0.2];
\node [align=center,align=center] at (1.2,0) {$2$};
\node [align=center,align=center] at (1.65,-0.15) {$,$};
\end{tikzpicture}\\
&\begin{tikzpicture}
\draw (0,0) circle [radius=0.2];
\node [align=center,align=center] at (0,0) {$1$};
\draw (0.6,0) circle [radius=0.2];
\node [align=center,align=center] at (0.6,0) {$1$};
\draw (0.2,0)--(0.4,0);
\node [align=center,align=center] at (1.2,0) {$\mapsto$};
\draw [dotted,thick](2.2-0.3,0) circle [radius=0.2];
\node [align=center,align=center] at (2.2-0.3,0) {$1$};
\draw [dotted,thick](2.8-0.3,0) circle [radius=0.2];
\node [align=center,align=center] at (2.8-0.3,0) {$1$};
\draw [dotted,thick](2.4-0.3,0)--(2.6-0.3,0);
\node [align=center,align=center] at (3,-0.15) {$,$};
\end{tikzpicture}\\
&\begin{tikzpicture}
\draw (0,0) circle [radius=0.2];
\node [align=center,align=center] at (0,0) {$1$};
\draw (0.6,0) circle [radius=0.2];
\node [align=center,align=center] at (0.6,0) {$0$};
\draw (0.2,0)--(0.4,0);
\draw (0.76,0.1) .. controls (1.1,0.2) and (1.1,-0.2) ..  (0.76,-0.1);
\node [align=center,align=center] at (1.5,0) {$\mapsto$};
\draw [dotted,thick](0+2.2,0) circle [radius=0.2];
\node [align=center,align=center] at (0+2.2,0) {$1$};
\draw [dotted,thick](0.6+2.2,0) circle [radius=0.2];
\node [align=center,align=center] at (0.6+2.2,0) {$0$};
\draw [dotted,thick](0.2+2.2,0)--(0.4+2.2,0);
\draw [dotted,thick](0.76+2.2,0.1) .. controls (1.1+2.2,0.2) and (1.1+2.2,-0.2) ..  (0.76+2.2,-0.1);
\node [align=center,align=center] at (3.55,-0.15) {$,$};
\end{tikzpicture}
\end{split}
\end{equation*}
where the right-hand sides are understood as linear combinations of graphs in the usual sense,
i.e., as elements in $\cV_{2,0}^c$.
\end{Example}

One of the main results in \cite{wz2} is the following duality theorem:
\begin{Theorem}
[\cite{wz2}]
For every $(g,n)$ with $2g-2+n>0$,
the duality map $\phi_{g,n}:\cV_{g,n}\to\cV_{g,n}$ is an involution,
i.e.,
$\phi_{g,n}^2 = \Id$.
\end{Theorem}

In particular,
denote by $\Ver_{g,n}\in \cV_{g,n}^c$ the stable vertex (in the usual sense)
of genus $g$ with $n$ external edges,
then by definition we know that
\be
\phi_{g,n} (\Ver_{g,n}) = n!\cdot \wcF_{g,n}
=n! \cdot \sum_{\Gamma\in\cG_{g,n}^c} \frac{1}{|\Aut(\Gamma)|} \Gamma
\ee
in $\cV_{g,n}^c$,
thus $\phi_{g,n}^2 (\Ver_{g,n}) = \Ver_{g,n}$ gives the following:
\begin{Theorem}
[\cite{wz2}]
For every $(g,n)$ with $2g-2+n>0$,
we have:
\be
\label{thm-pre-duality}
\sum_{\Gamma^\vee\in \cG_{g,n}^{\vee,c}}
\frac{1}{|\Aut(\Gamma^\vee)|} \Gamma^\vee =
\frac{1}{n!}\cdot \Ver_{g,n},
\ee
where $\cG_{g,n}^{\vee,c}$ is the set of all connected dotted stable graphs
of genus $g$ with $n$ external edges.
\end{Theorem}

\section{Generalized M\"obius Inversion Formula for Stable Graphs}
\label{sec-inversion-Mob}

In this section,
we introduce a partial ordering on the set $\cG_{g,n}^c$ of stable graphs
described by edge-contractions,
and construct a pair of functions $(\tilde\zeta, \tilde\mu)$ on $\cG_{g,n}^c$
such that they are inverse to each other in the incidence algebra.
These two functions are analogues of the zeta function and the M\"obius function in combinatorics respectively.
Moreover,
we derive an inversion formula using $\tilde\zeta$ and $\tilde\mu$,
which is analogous to the M\"obius inversion formula.

\subsection{Partially-ordered set, incidence algebra, and M\"obius inversion formula}
\label{sec-pre-mobius}

In this subsection we recall some basics of the incidence algebra and M\"obius inversion formula
for a locally finite partially-ordered set,
see Rota \cite{ro}.
We will follow the notations in \cite{drs}.

A partially-ordered set $P$ is called locally finite,
if for arbitrary two elements $x,y\in P$,
the segment $[x,y]:=\{z\in P| x\leq z\leq y\}$ is always finite.
Now fix a field $\bK$,
and let $P$ be a locally finite partially-ordered set $P$.
The incidence algebra $I(P,\bK)$ of $P$ over $\bK$ consists of all $\bK$-valued functions $f(x,y)$
($x,y\in P$) such that $f(x,y)=0$ unless $x\leq y$.
Given two functions $f,g\in I(P,\bK)$,
the product $h=f*g$ of them are defined by:
\be
\label{eq-def-incidence-product}
h(x,y):= \sum_{z\in [x,y]} f(x,z) g(z,y).
\ee
It is clear that the Kronecker delta
\be
\delta(x,y) = \begin{cases}
1, &\text{ if $x=y$;}\\
0, &\text{ if $x\not=y$.}
\end{cases}
\ee
is the identity element of the incidence algebra $I(P,\bK)$.

The zeta function $\zeta(x,y)$ of $P$ is defined by:
\be
\label{eq-def-zeta}
\zeta (x,y):= \begin{cases}
1, & \text{if $x\leq y$};\\
0, & \text{otherwise}.
\end{cases}
\ee
The zeta function is known to be invertible in the incidence algebra,
and the inverse is the so-called M\"obius function $\mu(x,y)$.
The M\"obius function can be constructed inductively as follows.
First,
define $\mu(x,x):=1$ for every $x\in P$.
Now suppose that $\mu(x,z)$ has been defined for every $z\in [x,y)$,
i,e, for every $z$ with $x\leq z <y$,
then one inductively defines:
\be
\mu(x,y) := -\sum_{x\leq z< y} \mu(x,z).
\ee
This summation is well-defined since $P$ is locally finite.
Then one can check that
\be
\zeta * \mu (x,y) =
\sum_{z\in [x,y]} \zeta(x,z)\mu(z,y) = \delta(x,y),
\ee
for every $x,y\in P$.
Moreover,
the following M\"obius inversion formula is known
(see Rota \cite[\S 3, Proposition 2]{ro}):
\begin{Theorem}
[\cite{ro}]
\label{thm-Mob-inv}
Let $P$ be a locally finite partially-ordered set,
and $f$ be a $\bK$-valued function on $P$.
And assume that there is an element $p\in P$ such that $f(x)=0$ unless $x\geq p$.
Define a $\bK$-valued function $g$ on $P$ by:
\be
\label{eq-def-g}
g(x):= \sum_{y\leq x} f(y)= \sum_{y\in P} f(y) \zeta(y,z) ,
\ee
then:
\be
f(x) = \sum_{y\leq x} g(y) \mu(y,x).
\ee
\end{Theorem}

\begin{Remark}
The existence of the above element $p\in P$ ensures that
\eqref{eq-def-g} is always a finite summation,
thus the function $g(x)$ is well-defined.
\end{Remark}

In what follows,
let us consider the sets $\cG_{g,n}^c$ of connected stable graphs.
We want to find an analogue of the above M\"obius inversion formula
such that it gives an interpretation of the duality theorem recalled in \S \ref{sec-pre-duality}.
We will see that in order to do this,
we need to modify the zeta function in a suitable manner such that
it encodes the orders of the automorphism groups of stable graphs.

\subsection{A partial ordering on the set of stable graphs}
\label{sec-cG-order}

In this subsection we introduce a partial ordering on the set $\cG_{g,n}^c$
of all connected stable graphs of genus $g$ with $n$ external edges.

Fix a pair of non-negative integers $(g,n)$ with $2g-2+n>0$,
and let us equip the set $\cG_{g,n}^c$ a partial ordering in the following way.
First we need to introduce the edge-contraction procedure for stable graphs.
An edge-contraction means one of the following procedures:
\begin{itemize}
\item[1)]
Remove a loop attached to a vertex $v$ of genus $g_v$,
and then replace this vertex by a vertex of genus $g_v+1$.
For examples,
\begin{equation*}
\begin{split}
&\begin{tikzpicture}
\draw (0,0) circle [radius=0.2];
\node [align=center,align=center] at (0,0) {$1$};
\draw (0.6,0) circle [radius=0.2];
\node [align=center,align=center] at (0.6,0) {$0$};
\draw (0.2,0)--(0.4,0);
\draw (0.76,0.1) .. controls (1.1,0.2) and (1.1,-0.2) ..  (0.76,-0.1);
\node [align=center,align=center] at (1.65,0) {$\mapsto$};
\draw (0+2.5,0) circle [radius=0.2];
\node [align=center,align=center] at (0+2.5,0) {$1$};
\draw (0.6+2.5,0) circle [radius=0.2];
\node [align=center,align=center] at (0.6+2.5,0) {$1$};
\draw (0.2+2.5,0)--(0.4+2.5,0);
\node [align=center,align=center] at (3.6,-0.1) {$;$};
\end{tikzpicture}\\
&\begin{tikzpicture}
\draw (1+1.4+0.2,0) circle [radius=0.2];
\draw (0.84+1.4+0.2,0.1) .. controls (0.5+1.4+0.2,0.2) and (0.5+1.4+0.2,-0.2) ..  (0.84+1.4+0.2,-0.1);
\node [align=center,align=center] at (1+1.4+0.2,0) {$1$};
\node [align=center,align=center] at (3.4,0) {$\mapsto$};
\draw (4.2,0) circle [radius=0.2];
\node [align=center,align=center] at (4.2,0) {$2$};
\node [align=center,align=center] at (4.7,-0.1) {$;$};
\end{tikzpicture}\\
&\begin{tikzpicture}
\draw (1+9,0) circle [radius=0.2];
\draw (0.4+9,0) circle [radius=0.2];
\draw (0.6+9,0)--(0.8+9,0);
\draw (1.16+9,0.1) .. controls (1.5+9,0.2) and (1.5+9,-0.2) ..  (1.16+9,-0.1);
\draw (0.24+9,0.1) .. controls (-0.1+9,0.2) and (-0.1+9,-0.2) ..  (0.24+9,-0.1);
\node [align=center,align=center] at (1+9,0) {$0$};
\node [align=center,align=center] at (0.4+9,0) {$0$};
\node [align=center,align=center] at (11,0) {$\mapsto$};
\draw (1+9+2.6,0) circle [radius=0.2];
\draw (0.4+9+2.6,0) circle [radius=0.2];
\draw (0.6+9+2.6,0)--(0.8+9+2.6,0);
\draw (0.24+9+2.6,0.1) .. controls (-0.1+9+2.6,0.2) and (-0.1+9+2.6,-0.2) ..  (0.24+9+2.6,-0.1);
\node [align=center,align=center] at (1+9+2.6,0) {$1$};
\node [align=center,align=center] at (0.4+9+2.6,0) {$0$};
\node [align=center,align=center] at (13.1,-0.12) {$.$};
\end{tikzpicture}
\end{split}
\end{equation*}
\item[2)]
Remove an internal edge $e$ which is not a loop,
and let the two (distinct) vertices $v_1,v_2$ (of genus $g_1,g_2$ respectively) joint by $e$
merges into a new vertex of genus $g_1+g_2$.
For examples,
\begin{equation*}
\begin{split}
&\begin{tikzpicture}
\draw (0,0) circle [radius=0.2];
\node [align=center,align=center] at (0,0) {$1$};
\draw (0.6,0) circle [radius=0.2];
\node [align=center,align=center] at (0.6,0) {$0$};
\draw (0.2,0)--(0.4,0);
\draw (0.76,0.1) .. controls (1.1,0.2) and (1.1,-0.2) ..  (0.76,-0.1);
\node [align=center,align=center] at (1.65,0) {$\mapsto$};
\draw (0+2.5,0) circle [radius=0.2];
\node [align=center,align=center] at (0+2.5,0) {$1$};
\draw (0.76+1.9,0.1) .. controls (1.1+1.9,0.2) and (1.1+1.9,-0.2) ..  (0.76+1.9,-0.1);
\node [align=center,align=center] at (3.2,-0.1) {$;$};
\end{tikzpicture}\\
&\begin{tikzpicture}
\draw (0,0) circle [radius=0.2];
\node [align=center,align=center] at (0,0) {$1$};
\draw (0.6,0) circle [radius=0.2];
\node [align=center,align=center] at (0.6,0) {$1$};
\draw (0.2,0)--(0.4,0);
\node [align=center,align=center] at (1.4,0) {$\mapsto$};
\draw (2.2,0) circle [radius=0.2];
\node [align=center,align=center] at (2.2,0) {$2$};
\node [align=center,align=center] at (2.7,-0.1) {$;$};
\end{tikzpicture}\\
&\begin{tikzpicture}
\draw (1+4.8,0) circle [radius=0.2];
\draw (0.84+4.8,0.1) .. controls (0.5+4.8,0.2) and (0.5+4.8,-0.2) ..  (0.84+4.8,-0.1);
\draw (1.16+4.8,0.1) .. controls (1.5+4.8,0.2) and (1.5+4.8,-0.2) ..  (1.16+4.8,-0.1);
\node [align=center,align=center] at (1+4.8,0) {$0$};
\node [align=center,align=center] at (4.8,0) {$\mapsto$};
\draw (1+10.2+0.2-8,0) circle [radius=0.2];
\draw (1.2+10.2+0.2-8,0)--(1.4+10.2+0.2-8,0);
\draw (1.16+10.2+0.2-8,0.1)--(1.44+10.2+0.2-8,0.1);
\draw (1.16+10.2+0.2-8,-0.1)--(1.44+10.2+0.2-8,-0.1);
\draw (1.6+10.2+0.2-8,0) circle [radius=0.2];
\node [align=center,align=center] at (1+10.2+0.2-8,0) {$0$};
\node [align=center,align=center] at (1.6+10.2+0.2-8,0) {$0$};
\node [align=center,align=center] at (12.4-5.9,-0.12) {$.$};
\end{tikzpicture}
\end{split}
\end{equation*}
\end{itemize}
It is clear that the edge-contraction procedures preserves
the genus and number of external edges of a stable graph.

\begin{Definition}
Define a partial ordering on the set $\cG_{g,n}^c$ as follows.
Given two graphs $\Gamma_1,\Gamma_2\in\cG_{g,n}^c$,
we require that $\Gamma_1\geq \Gamma_2$ if and only if
$\Gamma_1$ can be obtained from $\Gamma_2$ by contracting some internal edges successively.
\end{Definition}

Examples for small $(g,n)$ will be given in Appendix \ref{sec-app}.
The following is clear:
\begin{Lemma}
For each $(g,n)$ with $2g-2+n>0$,
we have $\Ver_{g,n} \geq \Gamma$ for every $\Gamma \in \cG_{g,n}^c$,
where $\Ver_{g,n} \in \cG_{g,n}^c$ is the stable vertex of genus $g$ with $n$ external edges
attached to it.
\end{Lemma}

And the following is a straightforward consequence of the stability condition:
\begin{Lemma}
Given $(g,n)$ with $2g-2+n>0$,
let $\Gamma\in \cG_{g,n}^c$ be a connected stable graph consisting of
some trivalent vertices of genus zero and edges connecting them,
then $\Gamma$ is minimal in $\cG_{g,n}^c$.
I.e.,
for every graph $\Gamma'\in \cG_{g,n}^c$,
we have either $\Gamma' \geq \Gamma$ or they are not comparable.
Conversely,
every minimal element in $\cG_{g,n}^c$ is of this form.
\end{Lemma}

Moreover,
from the definition of dotted stable graphs
one easily sees that the following property holds:
\begin{Lemma}
Let $\Gamma^\vee \in\cG_{g,n}^{\vee,c}$ be a connected dotted stable graph of genus $g$ with $n$ external edges,
and let $\Gamma\in \cG_{g,n}^c$ be the underlying graph in the usual sense.
I.e., let $\Gamma^\vee=\phi_{g,n}(\Gamma)$.
Then as an element in $\cV_{g,n}^c$,
$\Gamma^\vee$ is of the form:
\be
\label{eq-dotgr-form}
\Gamma^\vee = (-1)^{|E(\Gamma)|}\cdot \Gamma
+\sum_{\Gamma' < \Gamma} \text{coefficient}\cdot \Gamma'.
\ee
\end{Lemma}

The partially-ordered set $\cG_{g,n}^c$ is locally finite,
since $\cG_{g,n}^c$ is a finite set due to the stability condition \eqref{eq-stable-def}.
Now we can consider the incidence algebra $I(\cG_{g,n}^c,\bK)$ of $\cG_{g,n}^c$
for some field $\bK$,
and the product of two $\bK$-valued functions on $\cG_{g,n}^c$
is defined by \eqref{eq-def-incidence-product}.

\begin{Remark}
The partial ordering discussed above is inspired by
the stratification of the moduli space $\Mbar_{g,n}$ recalled in \S \ref{sec-pre-absqft}.
In fact,
given two graphs $\Gamma,\Gamma'\in \cG_{g,n}^c$ with $2g-2+n>0$,
one easily sees that $\Gamma' \leq \Gamma$ if and only if
$\cM_{\Gamma'} \subset \Mbar_{\Gamma}$
where $\Mbar_{\Gamma}$ is the closure of $\cM_\Gamma$ in $\Mbar_{g,n}$.

\end{Remark}

\begin{Remark}
We have introduced a partial ordering on $\cG_{g,n}^c$ for a fixed pair $(g,n)$.
One can also consider the set
\ben
\cG^c := \bigsqcup_{2g-2+n>0} \cG_{g,n}^c
\een
of all connected stable graphs,
and extend the partial ordering on each $\cG_{g,n}^c$ to $\cG^c$ by additionally requiring that
$\Gamma_1>\Gamma_2$ whenever $\Gamma_1\in\cG_{g_1,n_1}^c$, $\Gamma_2\in\cG_{g_2,n_2}^c$
such that $g_1>g_2$, or $g_1=g_2$ and $n_1>n_2$.
The ordering of the pairs $(g,n)$ will be useful when one considers the realizations,
see eg. \S \ref{sec-orb-intro}.
\end{Remark}

\subsection{Generalized zeta function and generalized M\"obius function}
\label{sec-gen-zeta-mu}

In this subsection we modify the zeta function \eqref{eq-def-zeta}
such that it encodes the information of the orders of automorphism groups of stable graphs.
Then we study its inverse in the incidence algebra.

Fix a pair $(g,n)$ with $2g-2+n>0$,
and let $\cG_{g,n}^c$ be equipped with the partial ordering
introduced in last subsection.
We introduce an analogue of the zeta function \eqref{eq-def-zeta}
on $\cG_{g,n}^c$ as follows:
\be
\tilde \zeta (\Gamma',\Gamma)
:= \begin{cases}
\frac{|\Aut(\Gamma)|}{|\Aut(\Gamma')|}\cdot |C(\Gamma',\Gamma)|,
& \text{if $\Gamma'\leq \Gamma$};\\
0, & \text{otherwise},
\end{cases}
\ee
where:
\begin{equation*}
C(\Gamma',\Gamma):=\{ E\subset E(\Gamma')\big|
\text{$\Gamma$ is obtained from $\Gamma'$ by
contracting edges in $E$} \}.
\end{equation*}
Here the subset $E\subset E(\Gamma)$ can be
$\emptyset$ (if $\Gamma=\Gamma'$) or $E(\Gamma)$ (if $\Gamma=\Ver_{g,n}$).

\begin{Example}
For example, let:
\begin{equation*}
\begin{tikzpicture}[scale=0.95]
\node [align=center,align=center] at (-0.6,0.015) {$\Gamma'=$};
\draw (1,0) circle [radius=0.2];
\draw (0.4,0) circle [radius=0.2];
\draw (0.6,0)--(0.8,0);
\draw (1.16,0.1) .. controls (1.5,0.2) and (1.5,-0.2) ..  (1.16,-0.1);
\draw (0.24,0.1) .. controls (-0.1,0.2) and (-0.1,-0.2) ..  (0.24,-0.1);
\node [align=center,align=center] at (1,0) {$0$};
\node [align=center,align=center] at (0.4,0) {$0$};
\node [align=center,align=center] at (1.65,-0.15) {$,$};
\node [align=center,align=center] at (-0.6+5,0.015) {$\Gamma=$};
\draw (1+5,0) circle [radius=0.2];
\draw (0.4+5,0) circle [radius=0.2];
\draw (0.6+5,0)--(0.8+5,0);
\draw (0.24+5,0.1) .. controls (-0.1+5,0.2) and (-0.1+5,-0.2) ..  (0.24+5,-0.1);
\node [align=center,align=center] at (1+5,0) {$1$};
\node [align=center,align=center] at (0.4+5,0) {$0$};
\node [align=center,align=center] at (1.4+5,-0.15) {$,$};
\end{tikzpicture}
\end{equation*}
then $|\Aut(\Gamma')|=8$ and $|\Aut(\Gamma)|=2$.
Notice that $\Gamma$ can be obtained from $\Gamma'$ by contracting
either one of the two loops in $\Gamma'$,
thus $|C(\Gamma', \Gamma)|=2$,
and
\begin{equation*}
\tilde\zeta(\Gamma',\Gamma)=\frac{2}{8}\times 2 = \half.
\end{equation*}
\end{Example}

\begin{Example}
Let $\Gamma=\Gamma'$.
We have $C(\Gamma,\Gamma)=\{\emptyset\}$, and then:
\be
\tilde\zeta (\Gamma',\Gamma')
=\frac{|\Aut(\Gamma')|}{|\Aut(\Gamma')|}\cdot 1
=1.
\ee
\end{Example}

\begin{Example}
Let $\Gamma=\Ver_{g,n}$.
We have $C(\Gamma',\Ver_{g,n})=\{E(\Gamma')\}$, and then:
\be
\tilde\zeta (\Gamma',\Ver_{g,n})
=\frac{|\Aut(\Gamma)|}{|\Aut(\Gamma')|}\cdot 1
= \frac{n!}{|\Aut(\Gamma')|}.
\ee
\end{Example}

Similar to the case in \S \ref{sec-pre-mobius},
the function $\tilde\zeta (x,y)$ is also invertible in the incidence algebra of $\cG_{g,n}^c$:

\begin{Proposition}
There exists a function $\tilde\mu(x,y)$ on $\cG_{g,n}^c$
such that:
\be
\label{eq-inverse-delta}
\sum_{y\in \cG_{g,n}^c} \tilde\zeta(x,y) \tilde\mu(y,z) = \delta(x,z),
\qquad
\forall x,z \in \cG_{g,n}^c.
\ee
\end{Proposition}
\begin{proof}
The inverse $\tilde\mu$ can be constructed inductively
using the same way of finding the M\"obius function $\mu$.
First,
let $\tilde \mu(x,z):=0$ unless $x\leq z$.
And for $x\leq z$,
we define $\tilde \mu(x,z)$ inductively on the number of elements in the interval $[x,z]$ as follows.
If $x=z$,
define $\tilde\mu(x,x):=1$;
and if $x<z$,
we inductively define:
\be
\label{eq-def-tildemu}
\tilde\mu(x,z) := -\sum_{y\in (x,z]} \tilde\zeta(x,y) \tilde\mu(y,z).
\ee
Then one can check that \eqref{eq-inverse-delta} holds:
\begin{equation*}
\sum_{y\in [x,z]} \tilde\zeta(x,y) \tilde\mu(y,z)
= \tilde\zeta(x,x) \tilde\mu(x,z) +
\sum_{y\in (x,z]} \tilde\zeta(x,y) \tilde\mu(y,z)
=0,
\end{equation*}
since $\tilde\zeta(x,x)=1$.
\end{proof}

We will call $\tilde\zeta(x,y)$ the generalized zeta function,
and call $\tilde\mu(x,y)$ the generalized M\"obius function.
Our main result in this subsection is the following:
\begin{Lemma}
\label{lem-graph-1}
Given a pair $(g,n)$ with $2g-2+n>0$
and $\Gamma\in \cG_{g,n}^c$,
we have:
\be
\tilde\mu(\Gamma,\Ver_{g,n})= \frac{(-1)^{|E(\Gamma)|}\cdot n!}{|\Aut(\Gamma)|},
\ee
where $\Ver_{g,n}$ is the stable vertex of genus g and valence $n$,
and $|E(\Gamma)|$ is the number of internal edges in $\Gamma$.
\end{Lemma}

\begin{proof}
We prove by induction on the number of elements in the interval $[\Gamma,\Ver_{g,n}]$.
First consider the case $\Gamma=\Ver_{g,n}$,
and we have:
\be
\frac{(-1)^{|E(\Ver_{g,n})|}\cdot n!}{|\Aut(\Ver_{g,n})|} = 1,
\ee
which matches with the definition $\tilde\mu(\Ver_{g,n},\Ver_{g,n})=1$.

Now consider the case $\Gamma< \Ver_{g,n}$.
By the definition \eqref{eq-def-tildemu} and the induction hypothesis we have:
\ben
\tilde\mu(\Gamma, \Ver_{g,n}) &=&
-\sum_{\Gamma'\in (\Gamma,\Ver_{g,n}]}  \tilde\zeta(\Gamma,\Gamma') \tilde\mu(\Gamma',\Ver_{g,n})\\
&=&
-\sum_{\Gamma'\in (\Gamma,\Ver_{g,n}]}  \frac{|\Aut(\Gamma')|}{|\Aut(\Gamma)|}|C(\Gamma,\Gamma')|
\cdot
\frac{(-1)^{|E(\Gamma')|}\cdot n!}{|\Aut(\Gamma')|}\\
&=& - \frac{n!}{|\Aut(\Gamma)|} \sum_{\Gamma'\in (\Gamma,\Ver_{g,n}]}
(-1)^{|E(\Gamma')|}\cdot |C(\Gamma,\Gamma')|,
\een
thus it suffices to prove that:
\ben
-\sum_{\Gamma'\in (\Gamma,\Ver_{g,n}]}
(-1)^{|E(\Gamma')|}\cdot |C(\Gamma,\Gamma')|
=(-1)^{|E(\Gamma)|},
\een
or equivalently,
\ben
\sum_{\Gamma'\in [\Gamma,\Ver_{g,n}]}
(-1)^{|E(\Gamma')|}\cdot |C(\Gamma,\Gamma')| =0,
\een
for every $\Gamma\in\cG_{g,n}^c$ with $\Gamma\not= \Ver_{g,n}$.
Recall that $|C(\Gamma',\Gamma)|$ is the number of ways to choose a subset $E\subset E(\Gamma)$
such that $\Gamma'$ is obtained from $\Gamma$ by contracting the edges in $E$,
thus the left-hand side of the above equation can be rewritten as:
\ben
\sum_{\Gamma'\in [\Gamma,\Ver_{g,n}]}
(-1)^{|E(\Gamma')|}\cdot |C(\Gamma,\Gamma')|
&=&\sum_{E\subset E(\Gamma)} (-1)^{|E(\Gamma)\backslash E|}\\
&=&(-1)^{E(\Gamma)}\cdot \sum_{E\subset E(\Gamma)} (-1)^{|E|}.
\een
Notice that $\Gamma \not= \Ver_{g,n}$ implies that $E(\Gamma)$ is a non-empty set,
thus
\ben
\sum_{E\subset E(\Gamma)} (-1)^{|E|}=0
\een
holds.
This completes the proof.
\end{proof}

\subsection{Generalized M\"obius inversion formula}

Now let us consider the analogue of the M\"obius inversion formula
(see Theorem \ref{thm-Mob-inv})
for such a pair of functions $(\tilde\zeta,\tilde\mu)$ on $\cG_{g,n}^c$.
The result is:
\begin{Theorem}
\label{thm-generalized-Mobius}
Let $\cG_{g,n}^c$ be the partially-ordered set of connected stable graphs
of genus $g$ with $n$ external edges.
Let $\tilde f:\cG_{g,n}^c\to \bR$ be an arbitrary function on $\cG_{g,n}^c$,
and let $\tilde g :\cG_{g,n}^c\to \bR$ be another function defined by:
\be
\label{eq-def-tilde-g-general}
\tilde g(x) :=\sum_{y \in \cG_{g,n}^c} \tilde f(y) \tilde\zeta(y,x)
=\sum_{y \leq x} \tilde f(y) \tilde\zeta(y,x).
\ee
Then we have:
\be
\label{eq-orb-Mob-inv}
\tilde f(x)= \sum_{y\in \cG_{g,n}^c} \tilde g(y) \tilde\mu(y,x)
=\sum_{y\leq x} \tilde g(y) \tilde\mu(y,x).
\ee
\end{Theorem}
\begin{proof}
This is a straightforward consequence of the fact that $\tilde\zeta * \tilde\mu =\delta$
in the incidence algebra.
In fact, by \eqref{eq-inverse-delta} we have:
\begin{equation*}
\sum_{y\leq x} \tilde g(y) \tilde\mu(y,x)
=\sum_{z\leq y \leq x} \tilde f(z) \tilde\zeta(z,y) \tilde\mu (y,x)
=\sum_{z\leq x} \tilde f(z) \delta(z,x) = \tilde f(x).
\end{equation*}
Thus the conclusion holds.
\end{proof}

\begin{Remark}
Here we do not need the special element $p$ as in Theorem \ref{thm-Mob-inv},
since $\cG_{g,n}^c$ is a finite set and thus the function $\tilde g$
given by \eqref{eq-def-tilde-g-general} is always well-defined.
\end{Remark}

\section{Realization of the Generalized M\"obius Inversion Formula}
\label{realization-duality}

In this subsection we consider the realization of the above inversion theorem
by assigning suitable Feynman rules to the stable graphs.
In this way we give an interpretation of the realization of the duality theorem \eqref{thm-pre-duality}
as a generalization of the M\"obius inversion formula.

\subsection{Assigning Feynman rules to stable graphs}

In this subsection we assign a particular Feynman rule to the stable graphs,
and choose the function $\tilde f$ to be the weight of a graph
with respect to this Feynman rule.

Let $\{F_{g,n}\}_{2g-2+n>0}$ be a family of formal variables
(formal functions, formal power series, etc.),
and let $\kappa$ be another formal variable.
We consider the following Feynman rule for every stable graph $\Gamma$:
\be
\begin{split}
&v\in V(\Gamma) \quad\mapsto\quad w_v:= F_{g_v,\val_v};\\
&e\in E(\Gamma) \quad\mapsto\quad w_e:=\kappa,
\end{split}
\ee
and define the weight of $\Gamma$ to be:
\be
\label{eq-FR-stablegraph}
w_\Gamma :=\prod_{v\in V(\Gamma)} w_v
\cdot \prod_{v\in E(\Gamma)} w_e
= \kappa^{|E(\Gamma)|}\cdot \prod_{v\in V(\Gamma)} F_{g_v,\val_v}.
\ee

Now fix a pair $(g,n)$ with $2g-2+n>0$,
and let $\cG_{g,n}^c$ be the set of all connected stable graphs
of genus $g$ with $n$ external edges,
equipped with the partial ordering discussed in \S \ref{sec-cG-order}.
Moreover,
we take the field $\bK$ to be an arbitrary field containing all polynomials in $\{F_{g,n}\}$ and $\kappa$.
Define the $\bK$-valued function $\tilde f$ on $\cG_{g,n}^c$ by
simply taking the weight of the graph with respect to the above Feynman rule:
\be
\label{eq-def-tildef}
\tilde f(\Gamma) := w_\Gamma,
\qquad \forall \Gamma\in \cG_{g,n}^c.
\ee
In particular,
we have:
\be
\tilde f(\Ver_{g,n}) = F_{g,n},
\ee
where $\Ver_{g,n}\in\cG_{g,n}$ is the stable vertex of genus $g$ and valence $n$.
Moreover, denote by $\Ver_{g,n}^\vee := \phi_{g,n}(\Ver_{g,n})$ the dotted stable vertex
of genus $g$ and valence $n$
(see \S \ref{sec-pre-dotted}),
then by definition we have:
\be
\label{eq-tildef-dualver}
\tilde f(\Ver_{g,n}^\vee) = n!\cdot \tilde f(\wcF_{g,n}) = n!\cdot \wF_{g,n},
\ee
where $\wF_{g,n}$ is the realization of the abstract $n$-point function $\wcF_{g,n}$
(see \S \ref{sec-pre-absqft})
with respect to the Feynman rule \eqref{eq-FR-stablegraph}:
\be
\label{eq-real-Fgn}
\wF_{g,n}:=
\sum_{\Gamma\in\cG_{g,n}^c} \frac{w_\Gamma}{|\Aut(\Gamma)|}
= \sum_{\Gamma\in\cG_{g,n}^c} \frac{\kappa^{|E(\Gamma)|}}{|\Aut(\Gamma)|}
\cdot \prod_{v\in V(\Gamma)} F_{g_v,\val_v}.
\ee

\subsection{Computation of the function $\tilde g$}

Now we define the function $\tilde g$ on the partially-ordered set $\cG_{g,n}^c$
using the formula \eqref{eq-def-tilde-g-general}:
\be
\label{eq-tildeg-real}
\tilde g(\Gamma):= \sum_{\Gamma'\in \cG_{g,n}^c} \tilde f(\Gamma')
\tilde\zeta(\Gamma',\Gamma)
=\sum_{\Gamma' \leq \Gamma} \frac{|\Aut(\Gamma)|}{|\Aut(\Gamma')|}
\cdot |C(\Gamma',\Gamma)| \cdot
w_{\Gamma'}.
\ee
where $\tilde f$ is given by \eqref{eq-def-tildef}.
Then it is clear that:
\be
\tilde g(\Ver_{g,n}) = \sum_{\Gamma' \leq \Ver_{g,n}}
\frac{n!}{|\Aut(\Gamma')|} \cdot w_{\Gamma'}
=n! \cdot \sum_{\Gamma'\in\cG_{g,n}^c}\frac{w_{\Gamma'}}{|\Aut(\Gamma)|}
= n!\cdot \wF_{g,n},
\ee
and thus by \eqref{eq-tildef-dualver} we have:
\be
\tilde g(\Ver_{g,n}) = \tilde f (\Ver_{g,n}^\vee).
\ee

Our main result in this subsection is the following:
\begin{Theorem}
\label{thm-tildeg-general}
For every $\Gamma\in \cG_{g,n}^c$, we have:
\be
\begin{split}
\tilde g(\Gamma) = &
\kappa^{|E(\Gamma)|}\cdot \prod_{v\in V(\Gamma)} \tilde g(\Ver_{g_v,\val_v})\\
=&\kappa^{|E(\Gamma)|}\cdot \prod_{v\in V(\Gamma)}
\bigg(\val_v!\cdot \wF_{g_v,\val_v}\bigg).
\end{split}
\ee
\end{Theorem}
\begin{proof}
This theorem is proved using the duality theorem (see \S \ref{sec-pre-duality}).

Given a stable graph $\Gamma\in \cG_{g,n}^c$,
let $\widehat g(\Gamma)$ be the following formal summation of stable graphs
in $\cG_{g,n}^c$:
\be
\label{eq-pf-def-hatg}
\widehat g(\Gamma):= \sum_{\Gamma' \leq \Gamma}
\frac{|\Aut(\Gamma)|}{|\Aut(\Gamma')|}\cdot|C(\Gamma', \Gamma)|\cdot \Gamma'
\in \cV_{g,n}^c.
\ee
Then $\tilde g(\Gamma)$ (defined by \eqref{eq-tildeg-real})
is obtained from the formal summation $\widehat g(\Gamma)$ by assigning the Feynman rule
\eqref{eq-FR-stablegraph}.

Consider the following weighted summation of $\widehat g(\Gamma)$ over all $\Gamma\in\cG_{g,n}^c$:
\begin{equation*}
\begin{split}
\sum_{\Gamma\in \cG_{g,n}^c}
\frac{(-1)^{|E(\Gamma)|}}{|\Aut(\Gamma)|}\widehat g(\Gamma)=&
\sum_{\Gamma\in\cG_{g,n}^c}\sum_{\Gamma'\leq \Gamma}
\frac{(-1)^{|E(\Gamma)|}}{|\Aut(\Gamma')|}\cdot|C(\Gamma', \Gamma)|\cdot \Gamma'\\
=& \sum_{\Gamma'\in\cG_{g,n}^c}
\bigg(\sum_{\Gamma\in[\Gamma',\Ver_{g,n}]}(-1)^{|E(\Gamma)|}
\cdot|C(\Gamma', \Gamma)| \bigg)
\frac{\Gamma'}{|\Aut(\Gamma')|}.
\end{split}
\end{equation*}
Using the same argument we've used in the proof of Lemma \ref{lem-graph-1},
we know that for a fixed $\Gamma'\in \cG_{g,n}^c$ with $\Gamma' \not= \Ver_{g,n}$,
the following identity holds:
\begin{equation*}
\sum_{\Gamma\in[\Gamma',\Ver_{g,n}]} (-1)^{|E(\Gamma)|} \cdot|C(\Gamma', \Gamma)|
=(-1)^{|E(\Gamma')|}\cdot \sum_{E\subset E(\Gamma')} (-1)^{|E|} =0,
\end{equation*}
therefore the above weighted summation becomes:
\be
\label{eq-pf-hat-g}
\begin{split}
&\sum_{\Gamma\in \cG_{g,n}^c}
\frac{(-1)^{|E(\Gamma)|}}{|\Aut(\Gamma)|}\widehat g(\Gamma)\\
=&
\bigg(\sum_{\Gamma\in[\Ver_{g,n},\Ver_{g,n}]}(-1)^{|E(\Gamma)|}
\cdot|C(\Ver_{g,n}, \Gamma)| \bigg) \cdot
\frac{\Ver_{g,n}}{|\Aut(\Ver_{g,n})|}\\
=& (-1)^0\cdot
|C(\Ver_{g,n}, \Ver_{g,n})|\cdot
\frac{\Ver_{g,n}}{|\Aut(\Ver_{g,n})|}\\
=&\frac{1}{n!}\cdot \Ver_{g,n}.
\end{split}
\ee
Now comparing the equality \eqref{eq-pf-hat-g}
with the duality theorem \eqref{thm-pre-duality} in the abstract QFT,
we obtain the following relation:
\begin{equation*}
\sum_{\Gamma\in \cG_{g,n}^c}
\frac{(-1)^{|E(\Gamma)|}}{|\Aut(\Gamma)|}\widehat g(\Gamma)
=\frac{1}{n!}\cdot \Ver_{g,n}=
\sum_{\Gamma^\vee\in\cG_{g,n}^{\vee,c}} \frac{1}{|\Aut(\Gamma^\vee)|} \Gamma^\vee,
\end{equation*}
i.e.,
\be
\label{eq-pf-relation-dual}
\sum_{\Gamma\in \cG_{g,n}^c}
\frac{(-1)^{|E(\Gamma)|}}{|\Aut(\Gamma)|}\widehat g(\Gamma)=
\sum_{\Gamma\in\cG_{g,n}^{c}} \frac{1}{|\Aut(\Gamma)|} \phi_{g,n}(\Gamma),
\ee
where $\phi_{g,n}$ is the duality map on $\cV_{g,n}^c$ which maps a graph in the usual sense
to a dotted graph of the same shape.
Then we claim that:
\be
\label{eq-pf-claim}
\widehat g(\Gamma) = (-1)^{|E(\Gamma)|}\cdot \phi_{g,n}(\Gamma),
\qquad \forall \Gamma\in\cG_{g,n}^c.
\ee

Now let us prove the claim \eqref{eq-pf-claim}.
First notice that $\cG_{g,n}^c$ is a partially ordered set
where $\Gamma\leq V_{g,n}$ for every $\Gamma\in\cG_{g,n}^c$;
and moreover,
$\widehat g(\Gamma)$ and $\phi_{g,n}(\Gamma)$ are both of the form
(see \eqref{eq-dotgr-form} and \eqref{eq-pf-def-hatg}):
\be
\label{eq-hatg-dot-form}
\pm \Gamma
+ \sum_{\Gamma' < \Gamma} \text{coefficient}\cdot \Gamma'.
\ee
Using this property,
we are able to prove \eqref{eq-pf-claim} inductively.
First consider the case $\Gamma = \Ver_{g,n}$.
By \eqref{eq-hatg-dot-form} we know that for every $\Gamma' \in\cG_{g,n}^c$ with $\Gamma'\not= \Ver_{g,n}$,
the linear combination of graph $\phi_{g,n}(\Gamma')\in\cV_{g,n}^c$ does not contain a nonzero multiple of $\Ver_{g,n}$
since $\Ver_{g,n}> \Gamma'$.
Then by comparing the coefficient of $\Ver_{g,n}$ in \eqref{eq-pf-relation-dual} we get:
\be
\widehat g(\Ver_{g,n}) = \phi_{g,n} (\Ver_{g,n}).
\ee
Now suppose that \eqref{eq-pf-claim} is true for all $\Gamma\in (\Gamma_0,\Ver_{g,n}]$.
By the induction hypothesis and the equality \eqref{eq-pf-relation-dual} we have:
\be
\sum_{\Gamma\in \cG_{g,n}^c\backslash (\Gamma_0,\Ver_{g,n}]}
\frac{(-1)^{|E(\Gamma)|}}{|\Aut(\Gamma)|}\widehat g(\Gamma)=
\sum_{\Gamma\in\cG_{g,n}^{c}\backslash (\Gamma_0,\Ver_{g,n}]}
\frac{1}{|\Aut(\Gamma)|} \phi_{g,n}(\Gamma).
\ee
Now consider the coefficient of $\Gamma_0$ in this equation,
we get:
\ben
\widehat g(\Gamma_0) = (-1)^{|E(\Gamma_0)|}\cdot \phi_{g,n}(\Gamma_0)
\een
again by using \eqref{eq-hatg-dot-form}.
Thus the claim \eqref{eq-pf-claim} is proved by induction.

Therefore,
when assigning the Feynman rule \eqref{eq-FR-stablegraph} to a stable graph $\Gamma$,
we get:
\be
\tilde g (\Gamma) = w_{\widehat g(\Gamma)}
= (-1)^{|E(\Gamma)|} \cdot w_{\phi_{g,n}(\Gamma)}.
\ee
Recall that in the construction of dotted stable graphs,
we glue dotted stable vertices (whose weight are given by $n!\cdot \wF_{g_v,\val_v}$)
together,
and multiply a factor $(-1)$ whenever a new internal edge is obtained.
Thus one easily sees that:
\begin{equation*}
\begin{split}
w_{\phi_{g,n}(\Gamma)}
=& (-\kappa)^{|E(\Gamma)|} \cdot \prod_{v\in V(\Gamma)}
w_{\phi_{g_v,\val_v}(\Ver_{g_v,\val_v})}\\
=&(-\kappa)^{|E(\Gamma)|}\cdot \prod_{v\in V(\Gamma)}
\bigg(\val_v!\cdot \wF_{g_v,\val_v}\bigg),
\end{split}
\end{equation*}
and then:
\begin{equation*}
\tilde g (\Gamma)
=\kappa^{|E(\Gamma)|}\cdot \prod_{v\in V(\Gamma)}
\bigg(\val_v!\cdot \wF_{g_v,\val_v}\bigg).
\end{equation*}
Now we have finished the proof.
\end{proof}

\subsection{Realization of the inversion formula}

Now we can use the above results to write down the realization of
the generalized M\"obius inversion formula.

Assign the Feynman rule \eqref{eq-FR-stablegraph} to stable graphs,
and define two function $\tilde f$ and $\tilde g$ on the locally finite partially-ordered set $\cG_{g,n}^c$ by:
\be
\label{eq-tilde-f-g}
\begin{split}
&\tilde f (\Gamma) := w_\Gamma,\\
&\tilde g(\Gamma):= \sum_{\Gamma'\leq \Gamma} \tilde f(\Gamma')
\tilde\zeta(\Gamma',\Gamma).
\end{split}
\ee
Then Theorem \ref{thm-generalized-Mobius} gives the following inversion formula:
\be
w_\Gamma=
\tilde f(\Gamma) = \sum_{\Gamma' \leq \Gamma}
\tilde g (\Gamma') \tilde \mu (\Gamma',\Gamma),
\ee
which represents the function $\tilde f$ in terms of $\tilde g$ and the generalized M\"obius function $\tilde\mu$
defined in \S \ref{sec-gen-zeta-mu}.

Moreover,
recall that in Lemma \ref{lem-graph-1} we have already found a closed formula
for $\tilde\mu(\Gamma,\Ver_{g,n})$.
Using this formula we have:
\begin{equation*}
F_{g,n}=
w_{\Ver_{g,n}} = \sum_{\Gamma'\leq \Ver_{g,n}} \tilde g(\Gamma') \tilde\mu(\Gamma,\Ver_{g,n})
= \sum_{\Gamma' \in\cG_{g,n}^c} \frac{(-1)^{|E(\Gamma)|}\cdot n!}{|\Aut(\Gamma)|}\cdot \tilde g(\Gamma').
\end{equation*}
Now denote by $\tF_{g,n}:= n!\cdot \wF_{g,n}$,
and then by Theorem \ref{thm-tildeg-general} we obtain
the following realization of the generalized M\"obius inversion formula:
\begin{Theorem}
[\cite{wz2}]
\label{thm-realization-inversion}
Let $\{F_{g,n}\}_{2g-2+n>0}$ and $\kappa$ be some formal variables,
and let $\tF_{g,n}$ be defined by the graph sum formula:
\be
\label{eq-graphsum-tildeF}
\tF_{g,n} := n! \cdot
\sum_{\Gamma\in \cG_{g,n}^c} \frac{\kappa^{|E(\Gamma)|}}{|\Aut(\Gamma)|}
\prod_{v\in V(\Gamma)} F_{g_v,\val_v}.
\ee
Then we have the following inversion formula:
\be
\label{eq-real-Mob-thm}
F_{g,n} = n! \cdot
\sum_{\Gamma\in \cG_{g,n}^c} \frac{(-\kappa)^{|E(\Gamma)|}}{|\Aut(\Gamma)|}
\prod_{v\in V(\Gamma)} \tF_{g_v,\val_v}.
\ee
\end{Theorem}

\begin{Remark}
In the previous work \cite{wz2} we have already obtained the above inversion formula
by assigning Feynman rules directly to the duality theorem for stable graphs,
and interpreted it as a pair of transformations on the space of theories
which are inverse to each other.
Here by a theory we simply mean a collection of formal variables $\{F_{g,n}\}_{2g-2+n>0}$
where $F_{g,n}$ is supposed to be formally understood as the $n$-point correlation function of genus $g$,
see \cite[\S 4]{wz2} for details.
Here we interpret this inversion formula in a fashion similar to
the well-known M\"obius inversion formula which is  fundamental in combinatorial analysis.
\end{Remark}

\subsection{Interpretation by formal Gaussian integrals and Fourier inversion theorem}
\label{sec-integral}

The realization \eqref{eq-real-Mob-thm} of the M\"obius inversion formula can also
be interpreted using the formal Gaussian integrals and the Fourier inverse theorem.
Now let us explain this.
All the computations in this subsection will be carries out formally
(as computations of formal power series).

It is well-known that after assigning the Feynman rule \eqref{eq-FR-stablegraph},
the transformation from $\{F_{g,n}\}$ to
the graph sums $\{\tF_{g,n}\}$ defined by \eqref{eq-graphsum-tildeF} can be
represented in terms of a one-dimensional formal Gaussian integral
(see eg. \cite[\S 4]{wz2}):
\be
\label{eq-formal-integ}
\begin{split}
&\exp\bigg(
\sum_{2g-2+n>0}
\lambda^{2g-2} \tF_{g,n}\cdot \frac{z^n}{n!}
\bigg)\\
= &
\frac{1}{(2\pi \lambda^2 \kappa)^{1/2}}
\int \exp\bigg(
\sum_{2g-2+ n>0}
\lambda^{2g-2} F_{g,n}\cdot \frac{x^n}{n!}
-\frac{\lambda^{-2}}{2\kappa} (x-z)^2
\bigg) dx.
\end{split}
\ee
Now let us denote:
\be
\tF_{g,n}^\vee :=
n! \cdot
\sum_{\Gamma\in \cG_{g,n}^c} \frac{(-\kappa)^{|E(\Gamma)|}}{|\Aut(\Gamma)|}
\prod_{v\in V(\Gamma)} \tF_{g_v,\val_v},
\ee
then again by the formal Gaussian integral representation we have:
\begin{equation*}
\begin{split}
&\exp\bigg(
\sum_{2g-2+n>0}
\lambda^{2g-2} \tF_{g,n}^\vee  \frac{y^n}{n!}
\bigg)\\
= &
\frac{1}{(-2\pi \lambda^2 \kappa)^{1/2}}
\int \exp\bigg(
\sum_{2g-2+ n>0}
\lambda^{2g-2} \tF_{g,n} \frac{z^n}{n!}
+\frac{\lambda^{-2}}{2\kappa} (z-y)^2
\bigg) dz\\
=&  \pm
\frac{1}{2\pi i \lambda^2 \kappa}
\int\int \exp\bigg(
\sum_{2g-2+ n>0}
\lambda^{2g-2} F_{g,n} \frac{x^n}{n!}
+\frac{\lambda^{-2}}{2\kappa} \big((z-y)^2
- (x-z)^2 \big)
\bigg) dxdz\\
=& \pm
\frac{e^{\frac{\lambda^{-2}}{2\kappa}y^2}}{2\pi i \lambda^2 \kappa}
\int\int \exp\bigg( -\frac{\lambda^{-2}}{2\kappa}x^2+
\sum_{2g-2+ n>0}
\lambda^{2g-2} F_{g,n} \frac{x^n}{n!}
+\frac{\lambda^{-2}}{\kappa} (xz-yz)\bigg) dxdz.
\end{split}
\end{equation*}
Now denote $\tilde z = \frac{\lambda^{-2}}{\kappa}\cdot iz$,
then the above formal integral becomes:
\begin{equation*}
\begin{split}
& \mp \frac{e^{\frac{\lambda^{-2}}{2\kappa}y^2}}{2\pi }
\int\int \exp\bigg( -\frac{\lambda^{-2}}{2\kappa}x^2+
\sum_{2g-2+ n>0}
\lambda^{2g-2} F_{g,n} \frac{x^n}{n!}\bigg) e^{-ix\tilde z} e^{iy\tilde z}
dxd\tilde z.
\end{split}
\end{equation*}
This is the composition of a Fourier transformation and an inverse Fourier transformation,
thus the Fourier inversion theorem gives:
\begin{equation*}
\begin{split}
&\exp\bigg(
\sum_{2g-2+n>0}
\lambda^{2g-2} \tF_{g,n}^\vee  \frac{y^n}{n!}
\bigg)\\
= & \mp \exp\bigg(\frac{\lambda^{-2}}{2\kappa}y^2\bigg)
\cdot
\exp\bigg( -\frac{\lambda^{-2}}{2\kappa}x^2+
\sum_{2g-2+ n>0}
\lambda^{2g-2} F_{g,n} \frac{x^n}{n!}\bigg)\\
=& \mp \exp\bigg(\sum_{2g-2+ n>0}
\lambda^{2g-2} F_{g,n} \frac{x^n}{n!}\bigg).
\end{split}
\end{equation*}
Then one sees that here the sign $\mp$ is supposed to be $+$,
and we have:
\be
\tF_{g,n}^\vee = F_{g,n},
\qquad 2g-2+n>0.
\ee
This gives an alternative proof of the inversion formula \eqref{eq-real-Mob-thm}.

\section{Open-Closed Duality of the Orbifold Euler Characteristics of $\cM_{g,n}$ and $\Mbar_{g,n}$}
\label{sec-O-C}

In this section we present an application of the generalized M\"obius inversion formula.
We show that Theorem \ref{thm-realization-inversion} gives an
open-closed duality of the orbifold Euler characteristics of the moduli spaces
$\cM_{g,n}$ and $\Mbar_{g,n}$.

\subsection{Orbifold Euler characteristics of $\cM_{g,n}$ and $\Mbar_{g,n}$}
\label{sec-orb-intro}

First let us recall some results of the computations of the
orbifold Euler characteristics of $\cM_{g,n}$ and $\Mbar_{g,n}$ in literatures.

Given a pair of non-negative integers $(g,n)$ with $2g-2+n>0$,
the orbifold Euler characteristic of the moduli space $\cM_{g,n}$
is given by the famous Harer-Zagier formula
(see Harer-Zagier \cite{hz}; Penner \cite{pe}):
\be
\chi(\cM_{g, n})=(-1)^{n} \cdot \frac{(2 g-1) B_{2 g}}{(2 g) !}(2 g+n-3) !, \quad 2 g-2+n>0,
\ee
where $B_{2g}$ is the $(2g)$-th Bernoulli number.
See also Kontsevich \cite[Appendix D]{kon1} for another proof.

The orbifold Euler characteristic of the Deligne-Mumford compactification $\Mbar_{g,n}$ of $\cM_{g,n}$
is computed using the stratification
of $\Mbar_{g,n}$ recalled in \S \ref{sec-pre-absqft}.
Notice that $\Mbar_{g,n}$ carries a natural structure of an orbifold,
and thus in the computation of its orbifold Euler characteristics
one needs to divide the contribution of each strata $\cM_\Gamma$ by
the order of the automorphism group $\Aut(\Gamma)$.
The orbifold Euler characteristics are given by the following summation over stable graphs
(see Bini-Harer \cite{bh}):
\be
\label{eq-chi-graph}
\chi(\Mbar_{g, n})=n!\cdot
\sum_{\Gamma \in \cG_{g, n}^{c}}
\frac{1}{|\Aut(\Gamma)|} \prod_{v \in V(\Gamma)}
\chi(\cM_{g_v, \val_v}).
\ee
\begin{Example}
Using this formula and the expressions in Example \ref{eg-abstract-fe},
one can compute first a few examples of $\chi(\Mbar_{g,n})$:
\begin{equation*}
\begin{split}
&\chi(\Mbar_{0,3}) = \chi(\cM_{0,3}) =1,\\
&\chi(\Mbar_{0,4}) = \chi(\cM_{0,4}) + 3\chi(\cM_{0,3})^2 =2,\\
&\chi(\Mbar_{0,5}) = \chi(\cM_{0,5}) + 10 \chi(\cM_{0,3})\chi(\cM_{0,4})
+ 15\chi(\cM_{0,3})^3 =7,
\end{split}
\end{equation*}
and:
\begin{equation*}
\begin{split}
\chi(\Mbar_{1,1}) =& \chi(\cM_{1,1}) + \half \chi(\cM_{0,3}) = \frac{5}{12},\\
\chi(\Mbar_{1,2}) =& \chi(\cM_{1,2}) + \chi(\cM_{1,1})\chi(\cM_{0,3}) + \half \chi(\cM_{0,4}) +
\chi(\cM_{0,3})^2 = \half,\\
\chi(\Mbar_{2,0}) =& \chi(\cM_{2,0}) + \half \chi(\cM_{1,2}) + \half \chi(\cM_{1,1})^2
+\half \chi(\cM_{1,1}) \chi(\cM_{0,3})\\
&+ \frac{1}{8} \chi(\cM_{0,4}) + \frac{5}{24} \chi(\cM_{0,3})^2 =\frac{119}{1440}.
\end{split}
\end{equation*}
\end{Example}

However,
it is not practical to use the above graph sum formula to carry out the specific computations
for larger $(g,n)$
due to the complexity of listing all possible graphs (without missing and repeating)
and then computing their automorphism groups.
In the previous work \cite{wz3},
the authors have used the formalisms developed in \cite{wz} and \cite{zhou1}
to derive various recursions and formulas for $\chi(\Mbar_{g,n})$.
Moreover,
by using these approaches we have related the computations of $\chi(\Mbar_{g,n})$
to other problems such as the Ramanujan polynomials,
topological 1D gravity, and the KP hierarchy,
see \cite{wz3} for details.

Another observation about the graph sum formula \eqref{eq-chi-graph} is that
it can be solved conversely such that
we are able to obtain $\{\chi(\cM_{g,n})\}$ once we know $\{\chi(\Mbar_{g,n})\}$.
This fact can be seen by introducing a total ordering on the set
\be
\{(g,n)\in\bZ^2 \big| g\geq 0,n\geq 0, 2g-2+n>0\}
\ee
by requiring $(g_1,n_1)>(g_2,n_2)$ if and only if $g_1>g_2$, or $g_1=g_2$ and $n_1>n_2$.
Then for every pair $(g,n)$,
the orbifold Euler characteristic $\chi(\Mbar_{g,n})$ is of the form:
\ben
\chi(\Mbar_{g,n}) = \chi(\cM_{g,n}) + P_{g,n},
\een
where $P_{g,n}$ is a polynomial of $\{\chi(\cM_{h,l})\}$ with $(h,l)<(g,n)$.
Thus by induction one easily finds that $\chi(\cM_{g,n})$ is of the form:
\ben
\chi(\Mbar_{g,n}) = \chi(\cM_{g,n}) + Q_{g,n},
\een
where $Q_{g,n}$ is a polynomial of $\{\chi(\Mbar_{h,l})\}$ with $(h,l)<(g,n)$.
For examples:
\begin{equation*}
\begin{split}
\chi(\cM_{0,3}) =& \chi(\Mbar_{0,3}),\\
\chi(\cM_{0,4}) =& \chi(\Mbar_{0,4}) - 3\chi(\cM_{0,3})^2\\
=& \chi(\Mbar_{0,4}) - 3\chi(\Mbar_{0,3})^2,\\
\chi(\cM_{0,5}) =& \chi(\Mbar_{0,5}) - 10 \chi(\cM_{0,3})\chi(\cM_{0,4})
- 15\chi(\cM_{0,3})^3\\
=& \chi(\Mbar_{0,5}) - 10 \chi(\Mbar_{0,3})\chi(\Mbar_{0,4})
+ 15\chi(\Mbar_{0,3})^3.
\end{split}
\end{equation*}

\subsection{Inversion formula and the open-closed duality}

The above inverse procedure of solving $\chi(\cM_{g,n})$ from $\{\chi(\Mbar_{h,l})\}_{(h,l)<(g,n)}$
is a special case of the results developed in
\S \ref{sec-inversion-Mob} and \S \ref{realization-duality}.
In fact,
one only needs to take
$\kappa=1$
and:
\be
F_{g,n} := \chi(\cM_{g,n}),
\qquad
2g-2+n>0.
\ee
in the Feynman rule \eqref{eq-FR-stablegraph}.
Now let $\bK=\bR$,
and let $\tilde f$ and $\tilde g$ be the two
real-valued functions on the set of stable graphs defined by \eqref{eq-tilde-f-g},
then:
\be
\tilde f(\Gamma) = \prod_{v\in V(\Gamma)} \chi(\cM_{g,n}),
\ee
and by \eqref{eq-chi-graph} and Theorem \ref{thm-tildeg-general} we have:
\be
\tilde g(\Gamma) = \prod_{v\in V(\Gamma)} \chi(\Mbar_{g,n}).
\ee
Then the generalized M\"obius inversion formula \eqref{eq-orb-Mob-inv} gives:
\begin{Theorem}
Given a pair $(g,n)$ with $2g-2+n>0$,
we have the following inversion formula for the graph sum \eqref{eq-chi-graph}:
\be
\chi(\cM_{g,n})=n!\cdot
\sum_{\Gamma \in \cG_{g, n}^{c}}
\frac{(-1)^{|E(\Gamma)|}}{|\Aut(\Gamma)|} \prod_{v \in V(\Gamma)}
\chi(\Mbar_{g_v,\val_v}).
\ee
\end{Theorem}

Now comparing this inversion formula with \eqref{eq-chi-graph},
and we see that the two graph sums are of exactly the same form.
The only differences are that we have switched the locations of
$\{\chi(\cM_{g,n})\}$ and $\{\chi(\Mbar_{g,n})\}$ and changed the sign of the propagator.
This gives a new example of the open-closed duality.

One can also represent this open-closed duality in terms of formal Gaussian integrals.
The following is known in \cite{bh}:
\begin{equation*}
\begin{split}
&\exp\bigg(
\sum_{2g-2+n>0}
\lambda^{2g-2} \chi(\Mbar_{g,n})\cdot \frac{z^n}{n!}
\bigg)\\
= &
\frac{1}{(2\pi \lambda^2 )^{1/2}}
\int \exp\bigg(
\sum_{2g-2+ n>0}
\lambda^{2g-2} \chi(\cM_{g,n})\cdot \frac{x^n}{n!}
-\frac{\lambda^{-2}}{2} (x-z)^2
\bigg) dx.
\end{split}
\end{equation*}
And by the discussions in \S \ref{sec-integral} we know that the inverse of this formula
is the following formal integral
(regarded as a equality for formal power series):
\begin{equation*}
\begin{split}
&\exp\bigg(
\sum_{2g-2+n>0}
\lambda^{2g-2} \chi(\cM_{g,n})\cdot \frac{z^n}{n!}
\bigg)\\
= &
-\frac{1}{(2\pi \lambda^2 )^{1/2}}
\int \exp\bigg(
\sum_{2g-2+ n>0}
\lambda^{2g-2} \chi(\Mbar_{g,n})\cdot \frac{x^n}{n!}
+\frac{\lambda^{-2}}{2} (x-z)^2
\bigg) dx.
\end{split}
\end{equation*}

\vspace{.2in}
{\bf Acknowledgements}.
The second named author is partly supported by NSFC (No.~12371254, No.~11890662, No.~12061131014).

\begin{appendices}

\section{Examples of the Partial Orderings}
\label{sec-app}

In this appendix,
we given some examples of the partial orderings on the set $\cG_{g,n}^c$
for small $g$ and $n$.

First consider $g=0$,
and in such cases the graphs are stable trees.
For simplicity we will use an (unmarked) solid dot to represent a vertex of genus zero.

\begin{Example}
Consider $(g,n)=(0,3)$.
This case is trivial since $\cG_{0,3}^c$ consists of only one element:
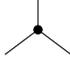
\begin{equation*}
\begin{tikzpicture}[scale=1.1]
\draw [fill] (0,0) circle [radius=0.05];
\draw (-0.4,-0.3)--(0,0);
\draw (0.4,-0.3)--(0,0);
\draw (0,0.4)--(0,0);
\end{tikzpicture}
\end{equation*}
\end{Example}

\begin{Example}
Consider $(g,n)=(0,4)$.
In this case $\cG_{0,4}^c$ is a totally ordered set consisting of two elements,
and we have:
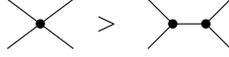
\begin{equation*}
\begin{tikzpicture}[scale=1.1]
\node [align=center,align=center] at (0.8,0) {$>$};
\draw [fill] (0,0) circle [radius=0.05];
\draw (-0.4,0.3)--(0.4,-0.3);
\draw (0.4,0.3)--(-0.4,-0.3);
\draw [fill] (1.6,0) circle [radius=0.05];
\draw [fill] (2,0) circle [radius=0.05];
\draw (1.6,0)--(1.3,0.3);
\draw (1.6,0)--(1.3,-0.3);
\draw (2.3,0.3)--(2,0);
\draw (2.3,-0.3)--(2,0);
\draw (1.6,0)--(2,0);
\end{tikzpicture}
\end{equation*}
\end{Example}

\begin{Example}
Consider $(g,n)=(0,5)$.
In this case In this case $\cG_{0,5}^c$ is a totally ordered set consisting of three elements,
and we have:
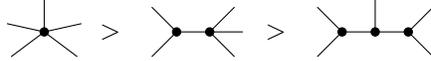
\begin{equation*}
\begin{tikzpicture}[scale=1.1]
\draw [fill] (0,0) circle [radius=0.05];
\draw (-0.45,0.1)--(0,0);
\draw (0.45,0.1)--(0,0);
\draw (0.4,-0.3)--(0,0);
\draw (-0.4,-0.3)--(0,0);
\draw (0,0.4)--(0,0);
\node [align=center,align=center] at (0.8,0) {$>$};
\draw [fill] (1.6,0) circle [radius=0.05];
\draw [fill] (2,0) circle [radius=0.05];
\draw (1.6,0)--(1.3,0.3);
\draw (1.6,0)--(1.3,-0.3);
\draw (2.3,0.3)--(2,0);
\draw (2.3,-0.3)--(2,0);
\draw (1.6,0)--(2.4,0);
\node [align=center,align=center] at (2.8,0) {$>$};
\draw [fill] (6.6-3,0) circle [radius=0.05];
\draw [fill] (7-3,0) circle [radius=0.05];
\draw [fill] (7.4-3,0) circle [radius=0.05];
\draw (7-3,0)--(7-3,0.4);
\draw (6.3-3,0.3)--(6.6-3,0);
\draw (6.3-3,-0.3)--(6.6-3,0);
\draw (7.7-3,0.3)--(7.4-3,0);
\draw (7.7-3,-0.3)--(7.4-3,0);
\draw (6.6-3,0)--(7.4-3,0);
\end{tikzpicture}
\end{equation*}
\end{Example}

\begin{Example}
Consider $(g,n)=(0,6)$.
For simplicity we will not write down all the relations,
but only list all pairs of adjacent graphs,
i.e., graphs $\Gamma,\Gamma'$ with $\Gamma>\Gamma'$ such that the open interval $(\Gamma',\Gamma)$ is empty.
We have:
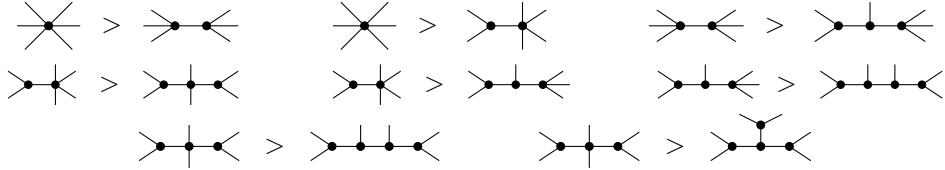
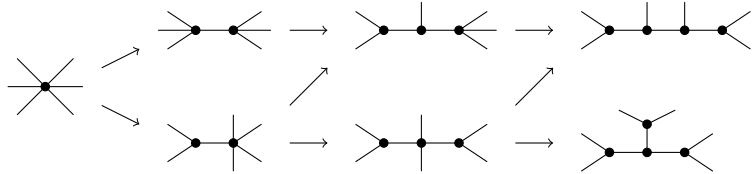
\begin{equation*}
\begin{tikzpicture}[scale=1.05]
\draw [fill] (0,0) circle [radius=0.05];
\draw (-0.4,0)--(0.4,0);
\draw (-0.3,0.3)--(0.3,-0.3);
\draw (0.3,0.3)--(-0.3,-0.3);
\node [align=center,align=center] at (0.8,0) {$>$};
\draw [fill] (1.6,0) circle [radius=0.05];
\draw [fill] (2,0) circle [radius=0.05];
\draw (1.6,0)--(1.3,0.2);
\draw (1.6,0)--(1.3,-0.2);
\draw (2.3,0.2)--(2,0);
\draw (2.3,-0.2)--(2,0);
\draw (1.2,0)--(2.4,0);
\draw [fill] (0+4,0) circle [radius=0.05];
\draw (-0.4+4,0)--(0.4+4,0);
\draw (-0.3+4,0.3)--(0.3+4,-0.3);
\draw (0.3+4,0.3)--(-0.3+4,-0.3);
\node [align=center,align=center] at (0.8+4,0) {$>$};
\draw [fill] (1.6+4,0) circle [radius=0.05];
\draw [fill] (2+4,0) circle [radius=0.05];
\draw (1.6+4,0)--(1.3+4,0.2);
\draw (1.6+4,0)--(1.3+4,-0.2);
\draw (2.3+4,0.2)--(2+4,0);
\draw (2.3+4,-0.2)--(2+4,0);
\draw (1.6+4,0)--(2+4,0);
\draw (6,-0.3)--(6,0.3);
\draw [fill] (1.6+6.4,0) circle [radius=0.05];
\draw [fill] (2+6.4,0) circle [radius=0.05];
\draw (1.6+6.4,0)--(1.3+6.4,0.2);
\draw (1.6+6.4,0)--(1.3+6.4,-0.2);
\draw (2.3+6.4,0.2)--(2+6.4,0);
\draw (2.3+6.4,-0.2)--(2+6.4,0);
\draw (1.2+6.4,0)--(2.4+6.4,0);
\node [align=center,align=center] at (9.2,0) {$>$};
\draw [fill] (6.6+3.4,0) circle [radius=0.05];
\draw [fill] (7+3.4,0) circle [radius=0.05];
\draw [fill] (7.4+3.4,0) circle [radius=0.05];
\draw (7+3.4,0)--(7+3.4,0.3);
\draw (6.3+3.4,0.2)--(6.6+3.4,0);
\draw (6.3+3.4,-0.2)--(6.6+3.4,0);
\draw (7.7+3.4,0.2)--(7.4+3.4,0);
\draw (7.7+3.4,-0.2)--(7.4+3.4,0);
\draw (6.6+3.4,0)--(7.8+3.4,0);
\end{tikzpicture}
\end{equation*}
\begin{equation*}
\begin{tikzpicture}[scale=0.9]
\draw [fill] (1.6+6.4-4.8,0) circle [radius=0.055];
\draw [fill] (2+6.4-4.8,0) circle [radius=0.055];
\draw (1.6+6.4-4.8,0)--(1.3+6.4-4.8,0.2);
\draw (1.6+6.4-4.8,0)--(1.3+6.4-4.8,-0.2);
\draw (2.3+6.4-4.8,0.2)--(2+6.4-4.8,0);
\draw (2.3+6.4-4.8,-0.2)--(2+6.4-4.8,0);
\draw (1.2+6.8-4.8,0)--(2.4+6-4.8,0);
\draw (2.4+6-4.8,-0.3)--(2.4+6-4.8,0.3);
\node [align=center,align=center] at (9.2-4.8,0) {$>$};
\draw [fill] (6.6+3.4-4.8,0) circle [radius=0.055];
\draw [fill] (7+3.4-4.8,0) circle [radius=0.055];
\draw [fill] (7.4+3.4-4.8,0) circle [radius=0.055];
\draw (7+3.4-4.8,-0.3)--(7+3.4-4.8,0.3);
\draw (6.3+3.4-4.8,0.2)--(6.6+3.4-4.8,0);
\draw (6.3+3.4-4.8,-0.2)--(6.6+3.4-4.8,0);
\draw (7.7+3.4-4.8,0.2)--(7.4+3.4-4.8,0);
\draw (7.7+3.4-4.8,-0.2)--(7.4+3.4-4.8,0);
\draw (6.6+3.4-4.8,0)--(7.8+3-4.8,0);
\draw [fill] (1.6+6.4,0) circle [radius=0.055];
\draw [fill] (2+6.4,0) circle [radius=0.055];
\draw (1.6+6.4,0)--(1.3+6.4,0.2);
\draw (1.6+6.4,0)--(1.3+6.4,-0.2);
\draw (2.3+6.4,0.2)--(2+6.4,0);
\draw (2.3+6.4,-0.2)--(2+6.4,0);
\draw (1.2+6.8,0)--(2.4+6,0);
\draw (8.4,-0.3)--(8.4,0.3);
\node [align=center,align=center] at (9.2,0) {$>$};
\draw [fill] (6.6+3.4,0) circle [radius=0.055];
\draw [fill] (7+3.4,0) circle [radius=0.055];
\draw [fill] (7.4+3.4,0) circle [radius=0.055];
\draw (7+3.4,0)--(7+3.4,0.3);
\draw (6.3+3.4,0.2)--(6.6+3.4,0);
\draw (6.3+3.4,-0.2)--(6.6+3.4,0);
\draw (7.7+3.4,0.2)--(7.4+3.4,0);
\draw (7.7+3.4,-0.2)--(7.4+3.4,0);
\draw (6.6+3.4,0)--(7.8+3.4,0);
\draw [fill] (6.6+3.4+2.8,0) circle [radius=0.055];
\draw [fill] (7+3.4+2.8,0) circle [radius=0.055];
\draw [fill] (7.4+3.4+2.8,0) circle [radius=0.055];
\draw (7+3.4+2.8,0)--(7+3.4+2.8,0.3);
\draw (6.3+3.4+2.8,0.2)--(6.6+3.4+2.8,0);
\draw (6.3+3.4+2.8,-0.2)--(6.6+3.4+2.8,0);
\draw (7.7+3.4+2.8,0.2)--(7.4+3.4+2.8,0);
\draw (7.7+3.4+2.8,-0.2)--(7.4+3.4+2.8,0);
\draw (6.6+3.4+2.8,0)--(7.8+3.4+2.8,0);
\node [align=center,align=center] at (7.4+3.4+3.6,0) {$>$};
\draw [fill] (15.2,0) circle [radius=0.055];
\draw [fill] (15.6,0) circle [radius=0.055];
\draw [fill] (16,0) circle [radius=0.055];
\draw [fill] (16.4,0) circle [radius=0.055];
\draw (15.2,0)--(16.4,0);
\draw (14.9,-0.2)--(15.2,0);
\draw (14.9,0.2)--(15.2,0);
\draw (16.7,-0.2)--(16.4,0);
\draw (16.7,0.2)--(16.4,0);
\draw (15.6,0)--(15.6,0.3);
\draw (16,0)--(16,0.3);
\end{tikzpicture}
\end{equation*}
\begin{equation*}
\begin{tikzpicture}[scale=0.95]
\draw [fill] (6.6+3.4+2.8,0) circle [radius=0.055];
\draw [fill] (7+3.4+2.8,0) circle [radius=0.055];
\draw [fill] (7.4+3.4+2.8,0) circle [radius=0.055];
\draw (7+3.4+2.8,-0.3)--(7+3.4+2.8,0.3);
\draw (6.3+3.4+2.8,0.2)--(6.6+3.4+2.8,0);
\draw (6.3+3.4+2.8,-0.2)--(6.6+3.4+2.8,0);
\draw (7.7+3.4+2.8,0.2)--(7.4+3.4+2.8,0);
\draw (7.7+3.4+2.8,-0.2)--(7.4+3.4+2.8,0);
\draw (6.6+3.4+2.8,0)--(7.8+3+2.8,0);
\node [align=center,align=center] at (7.4+3.4+3.6,0) {$>$};
\draw [fill] (15.2,0) circle [radius=0.055];
\draw [fill] (15.6,0) circle [radius=0.055];
\draw [fill] (16,0) circle [radius=0.055];
\draw [fill] (16.4,0) circle [radius=0.055];
\draw (15.2,0)--(16.4,0);
\draw (14.9,-0.2)--(15.2,0);
\draw (14.9,0.2)--(15.2,0);
\draw (16.7,-0.2)--(16.4,0);
\draw (16.7,0.2)--(16.4,0);
\draw (15.6,0)--(15.6,0.3);
\draw (16,0)--(16,0.3);
\draw [fill] (18.4,0) circle [radius=0.055];
\draw [fill] (7+3.4+2.8+5.6,0) circle [radius=0.055];
\draw [fill] (7.4+3.4+2.8+5.6,0) circle [radius=0.055];
\draw (7+3.4+2.8+5.6,-0.3)--(7+3.4+2.8+5.6,0.3);
\draw (6.3+3.4+2.8+5.6,0.2)--(6.6+3.4+2.8+5.6,0);
\draw (6.3+3.4+2.8+5.6,-0.2)--(6.6+3.4+2.8+5.6,0);
\draw (7.7+3.4+2.8+5.6,0.2)--(7.4+3.4+2.8+5.6,0);
\draw (7.7+3.4+2.8+5.6,-0.2)--(7.4+3.4+2.8+5.6,0);
\draw (6.6+3.4+2.8+5.6,0)--(7.8+3+2.8+5.6,0);
\node [align=center,align=center] at (20,0) {$>$};
\draw [fill] (20.8,0) circle [radius=0.055];
\draw [fill] (21.2,0) circle [radius=0.055];
\draw [fill] (21.6,0) circle [radius=0.055];
\draw [fill] (21.2,0.3) circle [radius=0.055];
\draw (21.2,0)--(21.2,0.3);
\draw (20.5,0.2)--(20.8,0);
\draw (20.5,-0.2)--(20.8,0);
\draw (21.9,0.2)--(21.6,0);
\draw (21.9,-0.2)--(21.6,0);
\draw (20.8,0)--(21.6,0);
\draw (20.9,0.45)--(21.2,0.3);
\draw (21.5,0.45)--(21.2,0.3);
\end{tikzpicture}
\end{equation*}
Or equivalently,
\begin{equation*}
\begin{tikzpicture}[scale=1.25]
\draw [fill] (0,0) circle [radius=0.045];
\draw (-0.4,0)--(0.4,0);
\draw (-0.3,0.3)--(0.3,-0.3);
\draw (0.3,0.3)--(-0.3,-0.3);
\draw [->] (0.6,0.2)--(1,0.4);
\draw [->] (0.6,-0.2)--(1,-0.4);
\draw [fill] (1.6,0+0.6) circle [radius=0.045];
\draw [fill] (2,0+0.6) circle [radius=0.045];
\draw (1.6,0+0.6)--(1.3,0.2+0.6);
\draw (1.6,0+0.6)--(1.3,-0.2+0.6);
\draw (2.3,0.2+0.6)--(2,0+0.6);
\draw (2.3,-0.2+0.6)--(2,0+0.6);
\draw (1.2,0+0.6)--(2.4,0+0.6);
\draw [->] (2.6,0.6)--(3,0.6);
\draw [->] (2.6,-0.6)--(3,-0.6);
\draw [->] (2.6,-0.2)--(3,0.2);
\draw [fill] (1.6,0-0.6) circle [radius=0.045];
\draw [fill] (2,0-0.6) circle [radius=0.045];
\draw (1.6,0-0.6)--(1.3,0.2-0.6);
\draw (1.6,0-0.6)--(1.3,-0.2-0.6);
\draw (2.3,0.2-0.6)--(2,0-0.6);
\draw (2.3,-0.2-0.6)--(2,0-0.6);
\draw (1.6,0-0.6)--(2,0-0.6);
\draw (2,0.3-0.6)--(2,-0.3-0.6);
\draw [fill] (3.6,0+0.6) circle [radius=0.045];
\draw [fill] (4,0+0.6) circle [radius=0.045];
\draw [fill] (4.4,0+0.6) circle [radius=0.045];
\draw (4,0+0.6)--(4,0.3+0.6);
\draw (3.3,0.2+0.6)--(3.6,0+0.6);
\draw (3.3,-0.2+0.6)--(3.6,0+0.6);
\draw (4.7,0.2+0.6)--(4.4,0+0.6);
\draw (4.7,-0.2+0.6)--(4.4,0+0.6);
\draw (3.6,0+0.6)--(4.8,0+0.6);
\draw [fill] (3.6,0-0.6) circle [radius=0.045];
\draw [fill] (4,0-0.6) circle [radius=0.045];
\draw [fill] (4.4,0-0.6) circle [radius=0.045];
\draw (4,-0.3-0.6)--(4,0.3-0.6);
\draw (3.3,0.2-0.6)--(3.6,0-0.6);
\draw (3.3,-0.2-0.6)--(3.6,0-0.6);
\draw (4.7,0.2-0.6)--(4.4,0-0.6);
\draw (4.7,-0.2-0.6)--(4.4,0-0.6);
\draw (3.6,0-0.6)--(4.4,0-0.6);
\draw [->] (5,0.6)--(5.4,0.6);
\draw [->] (5,-0.6)--(5.4,-0.6);
\draw [->] (5,-0.2)--(5.4,0.2);
\draw [fill] (3.6+2.4,0+0.6) circle [radius=0.045];
\draw [fill] (4+2.4,0+0.6) circle [radius=0.045];
\draw [fill] (4.4+2.4,0+0.6) circle [radius=0.045];
\draw [fill] (4.4+2.8,0+0.6) circle [radius=0.045];
\draw (4+2.4,0+0.6)--(4+2.4,0.3+0.6);
\draw (4+2.8,0+0.6)--(4+2.8,0.3+0.6);
\draw (3.3+2.4,0.2+0.6)--(3.6+2.4,0+0.6);
\draw (3.3+2.4,-0.2+0.6)--(3.6+2.4,0+0.6);
\draw (4.7+2.8,0.2+0.6)--(4.4+2.8,0+0.6);
\draw (4.7+2.8,-0.2+0.6)--(4.4+2.8,0+0.6);
\draw (3.6+2.4,0+0.6)--(4.8+2.4,0+0.6);
\draw [fill] (20.8-14.8,0-0.7) circle [radius=0.045];
\draw [fill] (21.2-14.8,0-0.7) circle [radius=0.045];
\draw [fill] (21.6-14.8,0-0.7) circle [radius=0.045];
\draw [fill] (21.2-14.8,0.3-0.7) circle [radius=0.045];
\draw (21.2-14.8,0-0.7)--(21.2-14.8,0.3-0.7);
\draw (20.5-14.8,0.2-0.7)--(20.8-14.8,0-0.7);
\draw (20.5-14.8,-0.2-0.7)--(20.8-14.8,0-0.7);
\draw (21.9-14.8,0.2-0.7)--(21.6-14.8,0-0.7);
\draw (21.9-14.8,-0.2-0.7)--(21.6-14.8,0-0.7);
\draw (20.8-14.8,0-0.7)--(21.6-14.8,0-0.7);
\draw (20.9-14.8,0.45-0.7)--(21.2-14.8,0.3-0.7);
\draw (21.5-14.8,0.45-0.7)--(21.2-14.8,0.3-0.7);
\end{tikzpicture}
\end{equation*}
where an arrow from $\Gamma$ to $\Gamma'$ means $\Gamma>\Gamma'$.
\end{Example}

\begin{Example}
Consider $(g,n)=(0,7)$.
In this case we have:
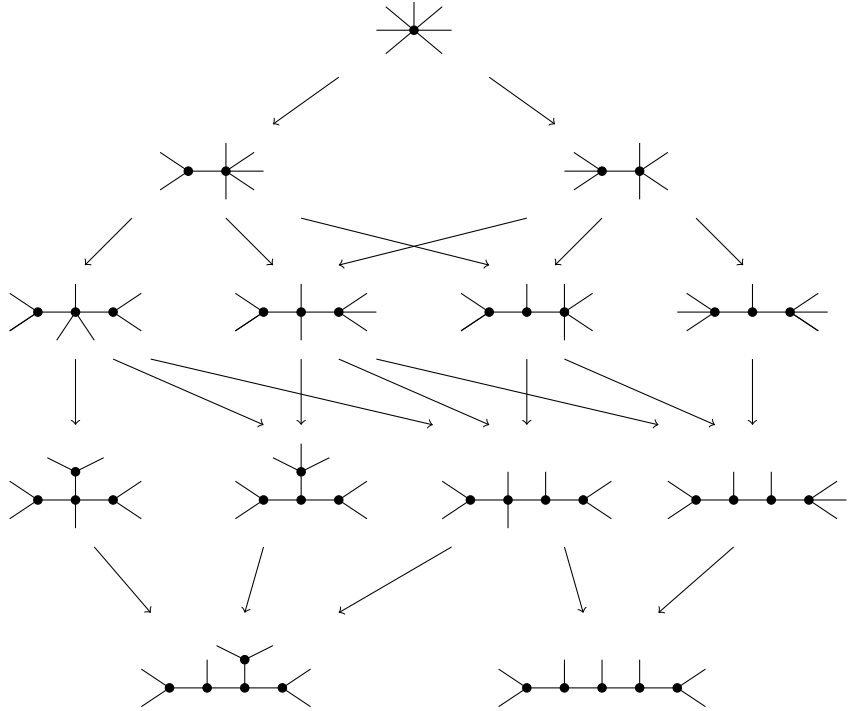
\begin{equation*}
\begin{tikzpicture}[scale=1.25]
\draw [fill] (0,0) circle [radius=0.045];
\draw (-0.4,0)--(0.4,0);
\draw (-0.3,0.25)--(0.3,-0.25);
\draw (0.3,0.25)--(-0.3,-0.25);
\draw (0,0)--(0,0.3);
\draw [->] (-0.8,-0.5)--(-1.5,-1);
\draw [->] (0.8,-0.5)--(1.5,-1);
\draw [fill] (-2,-1.5) circle [radius=0.045];
\draw [fill] (-2.4,-1.5) circle [radius=0.045];
\draw (-1.6,-1.5)--(-2.4,-1.5);
\draw (-2,-1.5)--(-1.7,-0.8-0.5);
\draw (-2,-1.5)--(-1.7,-1.2-0.5);
\draw (-2.4,-1.5)--(-2.7,-0.8-0.5);
\draw (-2.4,-1.5)--(-2.7,-1.2-0.5);
\draw (-2,-0.7-0.5)--(-2,-1.3-0.5);
\draw [fill] (2,-1-0.5) circle [radius=0.045];
\draw [fill] (2.4,-1-0.5) circle [radius=0.045];
\draw (1.6,-1-0.5)--(2.4,-1-0.5);
\draw (2,-1-0.5)--(1.7,-0.8-0.5);
\draw (2,-1-0.5)--(1.7,-1.2-0.5);
\draw (2.4,-1-0.5)--(2.7,-0.8-0.5);
\draw (2.4,-1-0.5)--(2.7,-1.2-0.5);
\draw (2.4,-0.7-0.5)--(2.4,-1.3-0.5);
\draw [->] (-3,-2)--(-3.5,-2.5);
\draw [->] (3,-2)--(3.5,-2.5);
\draw [->] (-2,-2)--(-1.5,-2.5);
\draw [->] (2,-2)--(1.5,-2.5);
\draw [->] (-1.2,-2)--(0.8,-2.5);
\draw [->] (1.2,-2)--(-0.8,-2.5);
\draw [fill] (-1.2,-3) circle [radius=0.045];
\draw [fill] (-1.6,-3) circle [radius=0.045];
\draw [fill] (-0.8,-3) circle [radius=0.045];
\draw (-4+2.4,-3)--(-4.3+2.4,-3.2);
\draw (-4+2.4,-3)--(-4.3+2.4,-2.8);
\draw (-3.2+2.4,-3)--(-2.9+2.4,-3.2);
\draw (-3.2+2.4,-3)--(-2.9+2.4,-2.8);
\draw (-4+2.4,-3)--(-4.3+2.4,-3.2);
\draw (-4+2.4,-3)--(-2.8+2.4,-3);
\draw (-3.6+2.4,-3.3)--(-3.6+2.4,-2.7);
\draw [fill] (1.2,-3) circle [radius=0.045];
\draw [fill] (1.6,-3) circle [radius=0.045];
\draw [fill] (0.8,-3) circle [radius=0.045];
\draw (-4+2.4+2.4,-3)--(-4.3+2.4+2.4,-3.2);
\draw (-4+2.4+2.4,-3)--(-4.3+2.4+2.4,-2.8);
\draw (-3.2+2.4+2.4,-3)--(-2.9+2.4+2.4,-3.2);
\draw (-3.2+2.4+2.4,-3)--(-2.9+2.4+2.4,-2.8);
\draw (-4+2.4+2.4,-3)--(-4.3+2.4+2.4,-3.2);
\draw (-4+2.4+2.4,-3)--(-3.2+2.4+2.4,-3);
\draw (-3.6+2.4+2.4,-3)--(-3.6+2.4+2.4,-2.7);
\draw (1.6,-3.3)--(1.6,-2.7);
\draw [fill] (-4,-3) circle [radius=0.045];
\draw [fill] (-3.2,-3) circle [radius=0.045];
\draw [fill] (-3.6,-3) circle [radius=0.045];
\draw (-4,-3)--(-4.3,-3.2);
\draw (-4,-3)--(-4.3,-2.8);
\draw (-3.2,-3)--(-2.9,-3.2);
\draw (-3.2,-3)--(-2.9,-2.8);
\draw (-4,-3)--(-4.3,-3.2);
\draw (-4,-3)--(-3.2,-3);
\draw (-3.6,-3)--(-3.6,-2.7);
\draw (-3.6,-3)--(-3.4,-3.3);
\draw (-3.6,-3)--(-3.8,-3.3);
\draw [fill] (4,-3) circle [radius=0.045];
\draw [fill] (3.6,-3) circle [radius=0.045];
\draw [fill] (3.2,-3) circle [radius=0.045];
\draw (4,-3)--(4.3,-3.2);
\draw (4,-3)--(4.3,-2.8);
\draw (3.2,-3)--(2.9,-3.2);
\draw (3.2,-3)--(2.9,-2.8);
\draw (4,-3)--(4.3,-3.2);
\draw (4.4,-3)--(2.8,-3);
\draw (3.6,-3)--(3.6,-2.7);
\draw [fill] (-4,-5) circle [radius=0.045];
\draw [fill] (-3.6,-5) circle [radius=0.045];
\draw [fill] (-3.2,-5) circle [radius=0.045];
\draw [fill] (-3.6,-4.7) circle [radius=0.045];
\draw (-4,-5)--(-3.2,-5);
\draw (-3.6,-5.3)--(-3.6,-4.7);
\draw (-4,-5)--(-4.3,-5.2);
\draw (-4,-5)--(-4.3,-4.8);
\draw (-3.2,-5)--(-2.9,-5.2);
\draw (-3.2,-5)--(-2.9,-4.8);
\draw (-3.6,-4.7)--(-3.9,-4.55);
\draw (-3.6,-4.7)--(-3.3,-4.55);
\draw [fill] (-1.2,-5) circle [radius=0.045];
\draw [fill] (-1.2,-4.7) circle [radius=0.045];
\draw [fill] (-1.6,-5) circle [radius=0.045];
\draw [fill] (-0.8,-5) circle [radius=0.045];
\draw (-4+2.4,-5)--(-3.2+2.4,-5);
\draw (-3.6+2.4,-5)--(-3.6+2.4,-4.4);
\draw (-4+2.4,-5)--(-4.3+2.4,-5.2);
\draw (-4+2.4,-5)--(-4.3+2.4,-4.8);
\draw (-3.2+2.4,-5)--(-2.9+2.4,-5.2);
\draw (-3.2+2.4,-5)--(-2.9+2.4,-4.8);
\draw (-3.6+2.4,-4.7)--(-3.9+2.4,-4.55);
\draw (-3.6+2.4,-4.7)--(-3.3+2.4,-4.55);
\draw [fill] (0.6,-5) circle [radius=0.045];
\draw [fill] (1,-5) circle [radius=0.045];
\draw [fill] (1.4,-5) circle [radius=0.045];
\draw [fill] (1.8,-5) circle [radius=0.045];
\draw (0.3,-5.2)--(0.6,-5);
\draw (0.3,-4.8)--(0.6,-5);
\draw (2.1,-5.2)--(1.8,-5);
\draw (2.1,-4.8)--(1.8,-5);
\draw (1.8,-5)--(0.6,-5);
\draw (1.4,-5)--(1.4,-4.7);
\draw (1,-5.3)--(1,-4.7);
\draw [fill] (3,-5) circle [radius=0.045];
\draw [fill] (3.4,-5) circle [radius=0.045];
\draw [fill] (3.8,-5) circle [radius=0.045];
\draw [fill] (4.2,-5) circle [radius=0.045];
\draw (0.3+2.4,-5.2)--(0.6+2.4,-5);
\draw (0.3+2.4,-4.8)--(0.6+2.4,-5);
\draw (2.1+2.4,-5.2)--(1.8+2.4,-5);
\draw (2.1+2.4,-4.8)--(1.8+2.4,-5);
\draw (2.2+2.4,-5)--(0.6+2.4,-5);
\draw (1.4+2.4,-5)--(1.4+2.4,-4.7);
\draw (1+2.4,-5)--(1+2.4,-4.7);
\draw [->] (-3.6,-3.5)--(-3.6,-4.2);
\draw [->] (3.6,-3.5)--(3.6,-4.2);
\draw [->] (-1.2,-3.5)--(-1.2,-4.2);
\draw [->] (1.2,-3.5)--(1.2,-4.2);
\draw [->] (-3.2,-3.5)--(-1.6,-4.2);
\draw [->] (-2.8,-3.5)--(0.2,-4.2);
\draw [->] (-0.8,-3.5)--(0.8,-4.2);
\draw [->] (-0.4,-3.5)--(2.6,-4.2);
\draw [->] (1.6,-3.5)--(3.2,-4.2);
\draw [fill] (-2.6,-7) circle [radius=0.045];
\draw [fill] (-2.2,-7) circle [radius=0.045];
\draw [fill] (-1.8,-7) circle [radius=0.045];
\draw [fill] (-1.4,-7) circle [radius=0.045];
\draw [fill] (-1.8,-6.7) circle [radius=0.045];
\draw (-2.6,-7)--(-1.4,-7);
\draw (-2.6,-7)--(-2.9,-7.2);
\draw (-2.6,-7)--(-2.9,-6.8);
\draw (-1.4,-7)--(-1.1,-6.8);
\draw (-1.4,-7)--(-1.1,-7.2);
\draw (-2.2,-7)--(-2.2,-6.7);
\draw (-1.8,-7)--(-1.8,-6.7);
\draw (-1.8,-6.7)--(-2.1,-6.55);
\draw (-1.8,-6.7)--(-1.5,-6.55);
\draw [fill] (2,-7) circle [radius=0.045];
\draw [fill] (1.6,-7) circle [radius=0.045];
\draw [fill] (1.2,-7) circle [radius=0.045];
\draw [fill] (2.4,-7) circle [radius=0.045];
\draw [fill] (2.8,-7) circle [radius=0.045];
\draw (1.2,-7)--(0.9,-7.2);
\draw (1.2,-7)--(0.9,-6.8);
\draw (2.8,-7)--(3.1,-6.8);
\draw (2.8,-7)--(3.1,-7.2);
\draw (1.2,-7)--(2.8,-7);
\draw (1.6,-7)--(1.6,-6.7);
\draw (2,-7)--(2,-6.7);
\draw (2.4,-7)--(2.4,-6.7);
\draw [->] (-3.4,-5.5)--(-2.8,-6.2);
\draw [->] (3.4,-5.5)--(2.6,-6.2);
\draw [->] (-1.6,-5.5)--(-1.8,-6.2);
\draw [->] (1.6,-5.5)--(1.8,-6.2);
\draw [->] (0.4,-5.5)--(-0.8,-6.2);
\end{tikzpicture}
\end{equation*}
where an arrow from $\Gamma$ to $\Gamma'$ means $\Gamma>\Gamma'$.
\end{Example}

Now let us consider the cases with $g=1$.

\begin{Example}
Consider $(g,n)=(1,1)$.
The set $\cG_{1,1}^c$ consists of two elements,
and we have:
\begin{equation*}
\begin{tikzpicture}
\draw (1,0) circle [radius=0.2];
\draw (1.2,0)--(1.5,0);
\node [align=center,align=center] at (1,0) {$1$};
\node [align=center,align=center] at (2,0) {$>$};
\draw (1+1.8,0) circle [radius=0.2];
\draw (1.2+1.8,0)--(1.5+1.8,0);
\draw (0.84+1.8,0.1) .. controls (0.5+1.8,0.2) and (0.5+1.8,-0.2) ..  (0.84+1.8,-0.1);
\node [align=center,align=center] at (1+1.8,0) {$0$};
\end{tikzpicture}
\end{equation*}
\end{Example}

\begin{Example}
Consider $(g,n)=(1,2)$.
The set $\cG_{1,2}^c$ consists of five elements,
and we have:
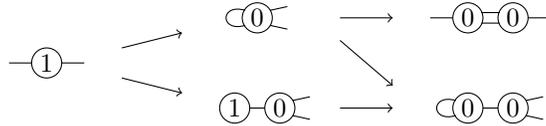
\begin{equation*}
\begin{tikzpicture}
\draw (1,0) circle [radius=0.2];
\draw (1.2,0)--(1.5,0);
\draw (0.5,0)--(0.8,0);
\draw [->] (2,0.2)--(2.8,0.4);
\draw [->] (2,-0.2)--(2.8,-0.4);
\node [align=center,align=center] at (1,0) {$1$};
\draw (1+1.8+1,0+0.6) circle [radius=0.2];
\draw (1.17+1.8+1,0.1+0.6)--(1.4+1.8+1,0.15+0.6);
\draw (1.17+1.8+1,-0.1+0.6)--(1.4+1.8+1,-0.15+0.6);
\draw (0.84+1.8+1,0.1+0.6) .. controls (0.5+1.8+1,0.2+0.6) and (0.5+1.8+1,-0.2+0.6) ..  (0.84+1.8+1,-0.1+0.6);
\node [align=center,align=center] at (1+1.8+1,0+0.6) {$0$};
\draw (1+3.9-0.8,0-0.6) circle [radius=0.2];
\draw (0.4+3.9-0.8,0-0.6) circle [radius=0.2];
\draw (1.17+3.9-0.8,0.1-0.6)--(1.4+3.9-0.8,0.15-0.6);
\draw (1.17+3.9-0.8,-0.1-0.6)--(1.4+3.9-0.8,-0.15-0.6);
\draw (0.6+3.9-0.8,0-0.6)--(0.8+3.9-0.8,0-0.6);
\draw [->] (4.9,0.6)--(5.6,0.6);
\draw [->] (4.9,0.3)--(5.6,-0.3);
\draw [->] (4.9,-0.6)--(5.6,-0.6);
\node [align=center,align=center] at (1+3.9-0.8,0-0.6) {$0$};
\node [align=center,align=center] at (0.4+3.9-0.8,0-0.6) {$1$};
\draw (1+6.2,0-0.6) circle [radius=0.2];
\draw (0.4+6.2,0-0.6) circle [radius=0.2];
\draw (1.17+6.2,0.1-0.6)--(1.4+6.2,0.15-0.6);
\draw (1.17+6.2,-0.1-0.6)--(1.4+6.2,-0.15-0.6);
\draw (0.6+6.2,0-0.6)--(0.8+6.2,0-0.6);
\node [align=center,align=center] at (1+6.3-0.1,0-0.6) {$0$};
\node [align=center,align=center] at (0.4+6.3-0.1,0-0.6) {$0$};
\draw (0.24+6.3-0.1,0.1-0.6) .. controls (-0.1+6.3-0.1,0.2-0.6) and (-0.1+6.3-0.1,-0.2-0.6) ..  (0.24+6.3-0.1,-0.1-0.6);
\draw (1+8.1-0.1-2.4,0+0.6) circle [radius=0.2];
\draw (0.5+8.1-0.1-2.4,0+0.6)--(0.8+8.1-0.1-2.4,0+0.6);
\draw (1.18+8.1-0.1-2.4,0.07+0.6)--(1.42+8.1-0.1-2.4,0.07+0.6);
\draw (1.18+8.1-0.1-2.4,-0.07+0.6)--(1.42+8.1-0.1-2.4,-0.07+0.6);
\draw (1.8+8.1-0.1-2.4,-0+0.6)--(2.1+8.1-0.1-2.4,0+0.6);
\draw (1.6+8.1-0.1-2.4,0+0.6) circle [radius=0.2];
\node [align=center,align=center] at (1+8.1-0.1-2.4,0+0.6) {$0$};
\node [align=center,align=center] at (1.6+8.1-0.1-2.4,0+0.6) {$0$};
\end{tikzpicture}
\end{equation*}
\end{Example}

\begin{Example}
Consider $(g,n)=(1,3)$.
We have:
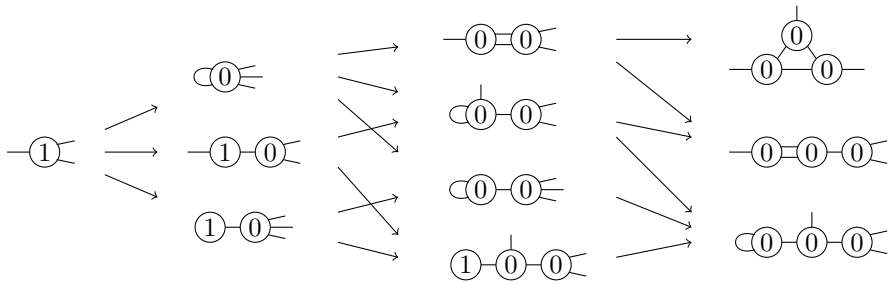
\begin{equation*}
\begin{tikzpicture}
\draw (1.6-0.4,0) circle [radius=0.2];
\draw (1.1-0.4,0)--(1.4-0.4,0);
\draw (1.76-0.4,0.1)--(2-0.4,0.15);
\draw (1.76-0.4,-0.1)--(2-0.4,-0.15);
\node [align=center,align=center] at (1.6-0.4,0) {$1$};
\draw [->](2,0)--(2.7,0);
\draw [->](2,0.3)--(2.7,0.6);
\draw [->](2,-0.3)--(2.7,-0.6);
\draw (1+2.6,0+1) circle [radius=0.2];
\draw (1.17+2.6,0.1+1)--(1.4+2.6,0.15+1);
\draw (1.17+2.6,-0.1+1)--(1.4+2.6,-0.15+1);
\draw (1.2+2.6,0+1)--(1.5+2.6,0+1);
\draw (0.84+2.6,0.1+1) .. controls (0.5+2.6,0.2+1) and (0.5+2.6,-0.2+1) ..  (0.84+2.6,-0.1+1);
\node [align=center,align=center] at (1+2.6,0+1) {$0$};
\draw (1+5-1.8,0) circle [radius=0.2];
\draw (0.4+5-1.8,0) circle [radius=0.2];
\draw (1.17+5-1.8,0.1)--(1.4+5-1.8,0.15);
\draw (1.17+5-1.8,-0.1)--(1.4+5-1.8,-0.15);
\draw (-0.1+5-1.8,0)--(0.2+5-1.8,0);
\draw (0.6+5-1.8,0)--(0.8+5-1.8,0);
\node [align=center,align=center] at (1+5-1.8,0) {$0$};
\node [align=center,align=center] at (0.4+5-1.8,0) {$1$};
\draw (1+3,0-1) circle [radius=0.2];
\draw (0.4+3,0-1) circle [radius=0.2];
\draw (1.17+3,0.1-1)--(1.4+3,0.15-1);
\draw (1.17+3,-0.1-1)--(1.4+3,-0.15-1);
\draw (0.6+3,0-1)--(0.8+3,0-1);
\draw (1.2+3,0-1)--(1.5+3,0-1);
\node [align=center,align=center] at (1+3,0-1) {$0$};
\node [align=center,align=center] at (0.4+3,0-1) {$1$};
\draw [->](5.1,1.3)--(5.9,1.4);
\draw [->](5.1,-1.2)--(5.9,-1.4);
\draw [->](5.1,1)--(5.9,0.8);
\draw [->](5.1,0.7)--(5.9,-0);
\draw [->](5.1,-0.8)--(5.9,-0.6);
\draw [->](5.1,0.2)--(5.9,0.4);
\draw [->](5.1,-0.2)--(5.9,-1.1);
\draw (1+6,0+1.5) circle [radius=0.2];
\draw (0.5+6,0+1.5)--(0.8+6,0+1.5);
\draw (1.18+6,0.07+1.5)--(1.42+6,0.07+1.5);
\draw (1.18+6,-0.07+1.5)--(1.42+6,-0.07+1.5);
\draw (1.76+6,0.1+1.5)--(2+6,0.15+1.5);
\draw (1.76+6,-0.1+1.5)--(2+6,-0.15+1.5);
\draw (1.6+6,0+1.5) circle [radius=0.2];
\node [align=center,align=center] at (1+6,0+1.5) {$0$};
\node [align=center,align=center] at (1.6+6,0+1.5) {$0$};
\draw (1+9.4-2.8,0+0.5) circle [radius=0.2];
\draw (0.4+9.4-2.8,0+0.5) circle [radius=0.2];
\draw (1.17+9.4-2.8,0.1+0.5)--(1.4+9.4-2.8,0.15+0.5);
\draw (1.17+9.4-2.8,-0.1+0.5)--(1.4+9.4-2.8,-0.15+0.5);
\draw (0.6+9.4-2.8,0+0.5)--(0.8+9.4-2.8,0+0.5);
\draw (0.4+9.4-2.8,0.2+0.5)--(0.4+9.4-2.8,0.4+0.5);
\draw (0.24+9.4-2.8,0.1+0.5) .. controls (-0.1+9.4-2.8,0.2+0.5) and (-0.1+9.4-2.8,-0.2+0.5) ..  (0.24+9.4-2.8,-0.1+0.5);
\node [align=center,align=center] at (1+9.4-2.8,0+0.5) {$0$};
\node [align=center,align=center] at (0.4+9.4-2.8,0+0.5) {$0$};
\draw (1+3+3.6,0-0.5) circle [radius=0.2];
\draw (0.4+3+3.6,0-0.5) circle [radius=0.2];
\draw (1.17+3+3.6,0.1-0.5)--(1.4+3+3.6,0.15-0.5);
\draw (1.17+3+3.6,-0.1-0.5)--(1.4+3+3.6,-0.15-0.5);
\draw (0.6+3+3.6,0-0.5)--(0.8+3+3.6,0-0.5);
\draw (1.2+3+3.6,0-0.5)--(1.5+3+3.6,0-0.5);
\draw (0.24+3+3.6,0.1-0.5) .. controls (-0.1+3+3.6,0.2-0.5) and (-0.1+3+3.6,-0.2-0.5) ..  (0.24+3+3.6,-0.1-0.5);
\node [align=center,align=center] at (1+3+3.6,0-0.5) {$0$};
\node [align=center,align=center] at (0.4+3+3.6,0-0.5) {$0$};
\draw (1+5.8+1.2,0-1.5) circle [radius=0.2];
\draw (0.4+5.8+1.2,0-1.5) circle [radius=0.2];
\draw (-0.2+5.8+1.2,0-1.5) circle [radius=0.2];
\draw (1.17+5.8+1.2,0.1-1.5)--(1.4+5.8+1.2,0.15-1.5);
\draw (1.17+5.8+1.2,-0.1-1.5)--(1.4+5.8+1.2,-0.15-1.5);
\draw (0+5.8+1.2,0-1.5)--(0.2+5.8+1.2,0-1.5);
\draw (0.6+5.8+1.2,0-1.5)--(0.8+5.8+1.2,0-1.5);
\draw (0.4+5.8+1.2,0.2-1.5)--(0.4+5.8+1.2,0.4-1.5);
\node [align=center,align=center] at (1+5.8+1.2,0-1.5) {$0$};
\node [align=center,align=center] at (0.4+5.8+1.2,0-1.5) {$0$};
\node [align=center,align=center] at (-0.2+5.8+1.2,0-1.5) {$1$};
\draw [->](8.8,1.5)--(9.8,1.5);
\draw [->](8.8,1.2)--(9.8,0.4);
\draw [->](8.8,0.4)--(9.8,0.2);
\draw [->](8.8,0.2)--(9.8,-0.8);
\draw [->](8.8,-0.6)--(9.8,-1);
\draw [->](8.8,-1.4)--(9.8,-1.2);
\draw (1+1.8+8,0-0.1+1.2) circle [radius=0.2];
\draw (1.8+1.8+8,0-0.1+1.2) circle [radius=0.2];
\draw (1.4+1.8+8,0.45-0.1+1.2) circle [radius=0.2];
\draw (1.2+1.8+8,0-0.1+1.2)--(1.6+1.8+8,0-0.1+1.2);
\draw (0.5+1.8+8,0-0.1+1.2)--(0.8+1.8+8,0-0.1+1.2);
\draw (2+1.8+8,0-0.1+1.2)--(2.3+1.8+8,0-0.1+1.2);
\draw (1.14+1.8+8,0.14-0.1+1.2)--(1.26+1.8+8,0.31-0.1+1.2);
\draw (1.66+1.8+8,0.14-0.1+1.2)--(1.54+1.8+8,0.31-0.1+1.2);
\node [align=center,align=center] at (1+1.8+8,0-0.1+1.2) {$0$};
\node [align=center,align=center] at (1.8+1.8+8,0-0.1+1.2) {$0$};
\node [align=center,align=center] at (1.4+1.8+8,0.45-0.1+1.2) {$0$};
\draw (1.4+1.8+8,0.65-0.1+1.2)--(1.4+1.8+8,0.85-0.1+1.2);
\draw (1+8.8+2.2,0) circle [radius=0.2];
\draw (0.4+8.8+2.2,0) circle [radius=0.2];
\draw (-0.2+8.8+2.2,0) circle [radius=0.2];
\draw (1.17+8.8+2.2,0.1)--(1.4+8.8+2.2,0.15);
\draw (1.17+8.8+2.2,-0.1)--(1.4+8.8+2.2,-0.15);
\draw (0.6+8.8+2.2,0)--(0.8+8.8+2.2,0);
\draw (-0.7+8.8+2.2,0)--(-0.4+8.8+2.2,0);
\draw (-0.02+8.8+2.2,0.07)--(0.22+8.8+2.2,0.07);
\draw (-0.02+8.8+2.2,-0.07)--(0.22+8.8+2.2,-0.07);
\node [align=center,align=center] at (1+8.8+2.2,0) {$0$};
\node [align=center,align=center] at (0.4+8.8+2.2,0) {$0$};
\node [align=center,align=center] at (-0.2+8.8+2.2,0) {$0$};
\draw (1+11,-1.2) circle [radius=0.2];
\draw (0.4+11,0-1.2) circle [radius=0.2];
\draw (-0.2+11,0-1.2) circle [radius=0.2];
\draw (1.17+11,0.1-1.2)--(1.4+11,0.15-1.2);
\draw (1.17+11,-0.1-1.2)--(1.4+11,-0.15-1.2);
\draw (0+11,0-1.2)--(0.2+11,0-1.2);
\draw (0.6+11,0-1.2)--(0.8+11,0-1.2);
\draw (0.4+11,0.2-1.2)--(0.4+11,0.4-1.2);
\draw (-0.36+11,0.1-1.2) .. controls (-0.7+11,0.2-1.2) and (-0.7+11,-0.2-1.2) ..  (-0.36+11,-0.1-1.2);
\node [align=center,align=center] at (1+11,0-1.2) {$0$};
\node [align=center,align=center] at (0.4+11,0-1.2) {$0$};
\node [align=center,align=center] at (-0.2+11,0-1.2) {$0$};
\end{tikzpicture}
\end{equation*}
\end{Example}

Then we consider some examples with $g=2$.

\begin{Example}
Consider $(g,n)=(2,0)$.
We have:
\begin{equation*}
\begin{tikzpicture}
\draw (0.4,0) circle [radius=0.2];
\node [align=center,align=center] at (0.4,0) {$2$};
\draw [->] (1,0.15)--(1.8,0.4);
\draw [->] (1,-0.15)--(1.8,-0.4);
\draw (1+1.4+0.2,0+0.5) circle [radius=0.2];
\draw (0.84+1.4+0.2,0.1+0.5) .. controls (0.5+1.4+0.2,0.2+0.5) and (0.5+1.4+0.2,-0.2+0.5) ..  (0.84+1.4+0.2,-0.1+0.5);
\node [align=center,align=center] at (1+1.4+0.2,0+0.5) {$1$};
\draw [->] (3.6,0.5)--(4.5,0.5);
\draw [->] (3.6,-0.5)--(4.5,-0.5);
\draw [->] (3.6,0.3)--(4.5,-0.3);
\draw (1+2.6+0.2-1.4,0-0.5) circle [radius=0.2];
\draw (1.6+2.6+0.2-1.4,0-0.5) circle [radius=0.2];
\draw (1.2+2.6+0.2-1.4,0-0.5)--(1.4+2.6+0.2-1.4,0-0.5);
\node [align=center,align=center] at (1+2.6+0.2-1.4,0-0.5) {$1$};
\node [align=center,align=center] at (1.6+2.6+0.2-1.4,0-0.5) {$1$};
\draw (1+4.8-0.3,0+0.5) circle [radius=0.2];
\draw (0.84+4.8-0.3,0.1+0.5) .. controls (0.5+4.8-0.3,0.2+0.5) and (0.5+4.8-0.3,-0.2+0.5) ..  (0.84+4.8-0.3,-0.1+0.5);
\draw (1.16+4.8-0.3,0.1+0.5) .. controls (1.5+4.8-0.3,0.2+0.5) and (1.5+4.8-0.3,-0.2+0.5) ..  (1.16+4.8-0.3,-0.1+0.5);
\node [align=center,align=center] at (1+4.8-0.3,0+0.5) {$0$};
\draw (1+6.8-2,0-0.5) circle [radius=0.2];
\draw (0.4+6.8-2,0-0.5) circle [radius=0.2];
\draw (0.6+6.8-2,0-0.5)--(0.8+6.8-2,0-0.5);
\draw (1.16+6.8-2,0.1-0.5) .. controls (1.5+6.8-2,0.2-0.5) and (1.5+6.8-2,-0.2-0.5) ..  (1.16+6.8-2,-0.1-0.5);
\node [align=center,align=center] at (1+6.8-2,0-0.5) {$0$};
\node [align=center,align=center] at (0.4+6.8-2,0-0.5) {$1$};
\draw (1+8,0-0.5) circle [radius=0.2];
\draw (0.4+8,0-0.5) circle [radius=0.2];
\draw (0.6+8,0-0.5)--(0.8+8,0-0.5);
\draw (1.16+8,0.1-0.5) .. controls (1.5+8,0.2-0.5) and (1.5+8,-0.2-0.5) ..  (1.16+8,-0.1-0.5);
\draw (0.24+8,0.1-0.5) .. controls (-0.1+8,0.2-0.5) and (-0.1+8,-0.2-0.5) ..  (0.24+8,-0.1-0.5);
\node [align=center,align=center] at (1+8,0-0.5) {$0$};
\node [align=center,align=center] at (0.4+8,0-0.5) {$0$};
\draw (1+10.2+0.2-3.2,0+0.5) circle [radius=0.2];
\draw (1.2+10.2+0.2-3.2,0+0.5)--(1.4+10.2+0.2-3.2,0+0.5);
\draw (1.16+10.2+0.2-3.2,0.1+0.5)--(1.44+10.2+0.2-3.2,0.1+0.5);
\draw (1.16+10.2+0.2-3.2,-0.1+0.5)--(1.44+10.2+0.2-3.2,-0.1+0.5);
\draw (1.6+10.2+0.2-3.2,0+0.5) circle [radius=0.2];
\node [align=center,align=center] at (1+10.2+0.2-3.2,0+0.5) {$0$};
\node [align=center,align=center] at (1.6+10.2+0.2-3.2,0+0.5) {$0$};
\draw [->] (6.6,0.5)--(7.5,0.5);
\draw [->] (6.6,-0.5)--(7.5,-0.5);
\draw [->] (6.6,0.3)--(7.5,-0.3);
\end{tikzpicture}
\end{equation*}
\end{Example}

\begin{Example}
Consider $(g,n)=(2,1)$.
We have:
\begin{equation*}
\begin{tikzpicture}
\draw (0.05,0) circle [radius=0.2];
\draw (0.25,0)--(0.55,0);
\node [align=center,align=center] at (0.05,0) {$2$};
\draw [->] (-0.4,-0.3)--(-1.2,-0.7);
\draw [->] (0.6,-0.3)--(1.4,-0.7);
\draw (1+1.8-4.8,0-1) circle [radius=0.2];
\draw (1.2+1.8-4.8,0-1)--(1.5+1.8-4.8,0-1);
\draw (0.84+1.8-4.8,0.1-1) .. controls (0.5+1.8-4.8,0.2-1) and (0.5+1.8-4.8,-0.2-1) ..  (0.84+1.8-4.8,-0.1-1);
\node [align=center,align=center] at (1+1.8-4.8,0-1) {$1$};
\draw (1+3.2-2.2,0-1) circle [radius=0.2];
\draw (1.2+3.2-2.2,0-1)--(1.4+3.2-2.2,0-1);
\draw (1.8+3.2-2.2,0-1)--(2.1+3.2-2.2,0-1);
\draw (1.6+3.2-2.2,0-1) circle [radius=0.2];
\node [align=center,align=center] at (1+3.2-2.2,0-1) {$1$};
\node [align=center,align=center] at (1.6+3.2-2.2,0-1) {$1$};
\draw [->] (-2.5,-1.4)--(-3.8,-2);
\draw [->] (-2.1,-1.4)--(-2.2,-2);
\draw [->] (-1.7,-1.4)--(-0.3,-2);
\draw [->] (-1.2,-1.4)--(2,-2);
\draw [->] (1.8,-1.4)--(0.5,-2);
\draw [->] (2.4,-1.4)--(2.7,-2);
\draw [->] (3,-1.4)--(4.5,-2);
\draw (1+2+2+1.2,0-2.5) circle [radius=0.2];
\draw (0.4+2+2+1.2,0-2.5) circle [radius=0.2];
\draw (-0.2+2+2+1.2,0-2.5) circle [radius=0.2];
\draw (0+2+2+1.2,0-2.5)--(0.2+2+2+1.2,0-2.5);
\draw (0.6+2+2+1.2,0-2.5)--(0.8+2+2+1.2,0-2.5);
\draw (0.4+2+2+1.2,0.2-2.5)--(0.4+2+2+1.2,0.4-2.5);
\node [align=center,align=center] at (1+2+2+1.2,0-2.5) {$1$};
\node [align=center,align=center] at (0.4+2+2+1.2,0-2.5) {$0$};
\node [align=center,align=center] at (-0.2+2+2+1.2,0-2.5) {$1$};
\draw (1-0.6+2.6,0-2.5) circle [radius=0.2];
\draw (0.4-0.6+2.6,0-2.5) circle [radius=0.2];
\draw (0.6-0.6+2.6,0-2.5)--(0.8-0.6+2.6,0-2.5);
\draw (1-0.6+2.6,0.2-2.5)--(1-0.6+2.6,0.4-2.5);
\draw (1.16-0.6+2.6,0.1-2.5) .. controls (1.5-0.6+2.6,0.2-2.5) and (1.5-0.6+2.6,-0.2-2.5) ..  (1.16-0.6+2.6,-0.1-2.5);
\node [align=center,align=center] at (1-0.6+2.6,0-2.5) {$0$};
\node [align=center,align=center] at (0.4-0.6+2.6,0-2.5) {$1$};
\draw (1+7.9-8.4,0-2.5) circle [radius=0.2];
\draw (0.4+7.9-8.4,0-2.5) circle [radius=0.2];
\draw (0.6+7.9-8.4,0-2.5)--(0.8+7.9-8.4,0-2.5);
\draw (-0.1+7.9-8.4,0-2.5)--(0.2+7.9-8.4,0-2.5);
\draw (1.16+7.9-8.4,0.1-2.5) .. controls (1.5+7.9-8.4,0.2-2.5) and (1.5+7.9-8.4,-0.2-2.5) ..  (1.16+7.9-8.4,-0.1-2.5);
\node [align=center,align=center] at (1+7.9-8.4,0-2.5) {$0$};
\node [align=center,align=center] at (0.4+7.9-8.4,0-2.5) {$1$};
\draw (1+9.4-13,0-2.5) circle [radius=0.2];
\draw (1.6+9.4-13,0-2.5) circle [radius=0.2];
\draw (1.18+9.4-13,0.07-2.5)--(1.42+9.4-13,0.07-2.5);
\draw (1.18+9.4-13,-0.07-2.5)--(1.42+9.4-13,-0.07-2.5);
\draw (1.8+9.4-13,-0-2.5)--(2.1+9.4-13,0-2.5);
\node [align=center,align=center] at (1+9.4-13,0-2.5) {$1$};
\node [align=center,align=center] at (1.6+9.4-13,0-2.5) {$0$};
\draw (1+5.6-11,-2.5) circle [radius=0.2];
\draw (1+5.6-11,0.2-2.5)--(1+5.6-11,0.4-2.5);
\draw (0.84+5.6-11,0.1-2.5) .. controls (0.5+5.6-11,0.2-2.5) and (0.5+5.6-11,-0.2-2.5) ..  (0.84+5.6-11,-0.1-2.5);
\draw (1.16+5.6-11,0.1-2.5) .. controls (1.5+5.6-11,0.2-2.5) and (1.5+5.6-11,-0.2-2.5) ..  (1.16+5.6-11,-0.1-2.5);
\node [align=center,align=center] at (1+5.6-11,0-2.5) {$0$};
\draw [->] (-4.5,-2.8)--(-4.5,-3.55);
\draw [->] (-4.1,-2.8)--(-2.7,-3.55);
\draw [->] (-2.2,-2.8)--(-2.2,-3.55);
\draw [->] (-1.5,-2.8)--(-0.3,-3.55);
\draw [->] (2.1,-2.8)--(0.8,-3.55);
\draw [->] (2.8,-2.8)--(2.8,-3.55);
\draw [->] (-3.7,-2.8)--(2,-3.55);
\draw [->] (0.3,-2.8)--(2.4,-3.55);
\draw [->] (1,-2.8)--(4.5,-3.55);
\draw [->] (3.5,-2.8)--(5.1,-3.55);
\draw [->] (5.6,-2.8)--(5.6,-3.55);
\draw (1-5.8,0-4) circle [radius=0.2];
\draw (1.2-5.8,0-4)--(1.4-5.8,0-4);
\draw (1.8-5.8,0-4)--(2.1-5.8,0-4);
\draw (1.16-5.8,0.1-4)--(1.44-5.8,0.1-4);
\draw (1.16-5.8,-0.1-4)--(1.44-5.8,-0.1-4);
\draw (1.6-5.8,0-4) circle [radius=0.2];
\node [align=center,align=center] at (1-5.8,0-4) {$0$};
\node [align=center,align=center] at (1.6-5.8,0-4) {$0$};
\draw (1,0-4) circle [radius=0.2];
\draw (0.4,0-4) circle [radius=0.2];
\draw (-0.2,0-4) circle [radius=0.2];
\draw (0.6,0-4)--(0.8,0-4);
\draw (-0.7,0-4)--(-0.4,0-4);
\draw (-0.02,0.07-4)--(0.22,0.07-4);
\draw (-0.02,-0.07-4)--(0.22,-0.07-4);
\node [align=center,align=center] at (1,0-4) {$1$};
\node [align=center,align=center] at (0.4,0-4) {$0$};
\node [align=center,align=center] at (-0.2,0-4) {$0$};
\draw (1+6.3-9.2,0-4) circle [radius=0.2];
\draw (0.4+6.3-9.2,0-4) circle [radius=0.2];
\draw (0.58+6.3-9.2,0.07-4)--(0.82+6.3-9.2,0.07-4);
\draw (0.58+6.3-9.2,-0.07-4)--(0.82+6.3-9.2,-0.07-4);
\draw (-0.1+6.3-9.2,0-4)--(0.2+6.3-9.2,0-4);
\draw (1.16+6.3-9.2,0.1-4) .. controls (1.5+6.3-9.2,0.2-4) and (1.5+6.3-9.2,-0.2-4) ..  (1.16+6.3-9.2,-0.1-4);
\node [align=center,align=center] at (1+6.3-9.2,0-4) {$0$};
\node [align=center,align=center] at (0.4+6.3-9.2,0-4) {$0$};
\draw (1+4-1.8,0-4) circle [radius=0.2];
\draw (0.4+4-1.8,0-4) circle [radius=0.2];
\draw (0.6+4-1.8,0-4)--(0.8+4-1.8,0-4);
\draw (1+4-1.8,0.2-4)--(1+4-1.8,0.4-4);
\draw (1.16+4-1.8,0.1-4) .. controls (1.5+4-1.8,0.2-4) and (1.5+4-1.8,-0.2-4) ..  (1.16+4-1.8,-0.1-4);
\draw (0.24+4-1.8,0.1-4) .. controls (-0.1+4-1.8,0.2-4) and (-0.1+4-1.8,-0.2-4) ..  (0.24+4-1.8,-0.1-4);
\node [align=center,align=center] at (1+4-1.8,0-4) {$0$};
\node [align=center,align=center] at (0.4+4-1.8,0-4) {$0$};
\draw (1+3.5+1.9,0-4) circle [radius=0.2];
\draw (0.4+3.5+1.9,0-4) circle [radius=0.2];
\draw (-0.2+3.5+1.9,0-4) circle [radius=0.2];
\draw (0+3.5+1.9,0-4)--(0.2+3.5+1.9,0-4);
\draw (0.6+3.5+1.9,0-4)--(0.8+3.5+1.9,0-4);
\draw (0.4+3.5+1.9,0.2-4)--(0.4+3.5+1.9,0.4-4);
\node [align=center,align=center] at (1+3.5+1.9,0-4) {$1$};
\node [align=center,align=center] at (0.4+3.5+1.9,0-4) {$0$};
\node [align=center,align=center] at (-0.2+3.5+1.9,0-4) {$0$};
\draw (-0.36+3.5+1.9,0.1-4) .. controls (-0.7+3.5+1.9,0.2-4) and (-0.7+3.5+1.9,-0.2-4) ..  (-0.36+3.5+1.9,-0.1-4);
\draw [->] (-4.4,-4.35)--(-3,-5.15);
\draw [->] (-2.3,-4.35)--(-2.3,-5.15);
\draw [->] (0.3,-4.35)--(0.3,-5.15);
\draw [->] (-1.9,-4.35)--(-0.2,-5.15);
\draw [->] (2.6,-4.35)--(0.8,-5.15);
\draw [->] (3.1,-4.35)--(3.1,-5.15);
\draw [->] (5.5,-4.35)--(3.9,-5.15);
\draw (1+2-6,0-0.1-6) circle [radius=0.2];
\draw (1.18+2-6,0.07-0.1-6)--(1.62+2-6,0.07-0.1-6);
\draw (1.18+2-6,-0.07-0.1-6)--(1.62+2-6,-0.07-0.1-6);
\draw (1.8+2-6,0-0.1-6) circle [radius=0.2];
\draw (1.4+2-6,0.45-0.1-6) circle [radius=0.2];
\draw (1.14+2-6,0.14-0.1-6)--(1.26+2-6,0.31-0.1-6);
\draw (1.66+2-6,0.14-0.1-6)--(1.54+2-6,0.31-0.1-6);
\node [align=center,align=center] at (1+2-6,0-0.1-6) {$0$};
\node [align=center,align=center] at (1.8+2-6,0-0.1-6) {$0$};
\draw (1.4+2-6,0.65-0.1-6)--(1.4+2-6,0.85-0.1-6);
\node [align=center,align=center] at (1.4+2-6,0.45-0.1-6) {$0$};
\draw (1+8.9-9,-5.8) circle [radius=0.2];
\draw (0.4+8.9-9,-5.8) circle [radius=0.2];
\draw (-0.2+8.9-9,-5.8) circle [radius=0.2];
\draw (0.6+8.9-9,0-5.8)--(0.8+8.9-9,0-5.8);
\draw (-0.7+8.9-9,0-5.8)--(-0.4+8.9-9,0-5.8);
\draw (-0.02+8.9-9,0.07-5.8)--(0.22+8.9-9,0.07-5.8);
\draw (-0.02+8.9-9,-0.07-5.8)--(0.22+8.9-9,-0.07-5.8);
\node [align=center,align=center] at (1+8.9-9,0-5.8) {$0$};
\node [align=center,align=center] at (0.4+8.9-9,0-5.8) {$0$};
\node [align=center,align=center] at (-0.2+8.9-9,0-5.8) {$0$};
\draw (1.16+8.9-9,0.1-5.8) .. controls (1.5+8.9-9,0.2-5.8) and (1.5+8.9-9,-0.2-5.8) ..  (1.16+8.9-9,-0.1-5.8);
\draw (1+2.6,0-5.8) circle [radius=0.2];
\draw (0.4+2.6,0-5.8) circle [radius=0.2];
\draw (1.6+2.6,0-5.8) circle [radius=0.2];
\draw (0.6+2.6,0-5.8)--(0.8+2.6,0-5.8);
\draw (1.2+2.6,0-5.8)--(1.4+2.6,0-5.8);
\draw (1+2.6,0.2-5.8)--(1+2.6,0.4-5.8);
\draw (1.76+2.6,0.1-5.8) .. controls (2.1+2.6,0.2-5.8) and (2.1+2.6,-0.2-5.8) ..  (1.76+2.6,-0.1-5.8);
\draw (0.24+2.6,0.1-5.8) .. controls (-0.1+2.6,0.2-5.8) and (-0.1+2.6,-0.2-5.8) ..  (0.24+2.6,-0.1-5.8);
\node [align=center,align=center] at (1+2.6,0-5.8) {$0$};
\node [align=center,align=center] at (0.4+2.6,0-5.8) {$0$};
\node [align=center,align=center] at (1.6+2.6,0-5.8) {$0$};
\end{tikzpicture}
\end{equation*}
\end{Example}

\end{appendices}

\end{document}